\documentclass[11pt,leqno]{article}
\usepackage{amsfonts}
\usepackage{amsthm,amsfonts,amssymb,amsmath,oldgerm, graphicx}
\usepackage{epsfig}
\numberwithin{equation}{section}
\usepackage[thinlines]{easybmat}

\renewcommand\d{\partial}

\def\eps{\varepsilon }

\renewcommand\d{\partial}

\def\eps{\varepsilon}

\newcommand\br{\begin{remark}}
\newcommand\er{\end{remark}}
\newcommand\bp{\begin{pmatrix}}
\newcommand\ep{\end{pmatrix}}
\newcommand\be{\begin{equation}}
\newcommand\ee{\end{equation}}
\newcommand\ba{\begin{equation}\begin{aligned}}
\newcommand\ea{\end{aligned}\end{equation}}

\newcommand{\bap}{\begin{app}}
\newcommand{\eap}{\end{app}}
\newcommand{\begs}{\begin{exams}}
\newcommand{\eegs}{\end{exams}}
\newcommand{\beg}{\begin{example}}
\newcommand{\eeg}{\end{exaplem}}
\newcommand{\bpr}{\begin{proposition}}
\newcommand{\epr}{\end{proposition}}
\newcommand{\bt}{\begin{theorem}}
\newcommand{\et}{\end{theorem}}
\newcommand{\bc}{\begin{corollary}}
\newcommand{\ec}{\end{corollary}}
\newcommand{\bl}{\begin{lemma}}
\newcommand{\el}{\end{lemma}}
\newcommand{\bd}{\begin{definition}}
\newcommand{\ed}{\end{definition}}
\newcommand{\brs}{\begin{remarks}}
\newcommand{\ers}{\end{remarks}}

\newtheorem{theo}{Theorem}[section]

\newtheorem{lem}[theo]{Lemma}

\newtheorem{exams}[theo]{Examples}

\numberwithin{equation}{section}
\newcommand{\D }{\mathcal{D}}

\newcommand{\const}{\text{\rm constant}}

\newtheorem{theorem}{Theorem}[section]
\newtheorem{proposition}[theorem]{Proposition}
\newtheorem{corollary}[theorem]{Corollary}
\newtheorem{lemma}[theorem]{Lemma}
\newtheorem{definition}[theorem]{Definition}

\newtheorem{example}[theorem]{Example}
\newtheorem{remark}[theorem]{Remark}

\newcommand\cA{{\cal  A}}

\newcommand\cI{{\cal  I}}

\newcommand{\RM}{\mathbb{R}}

\newcommand{\beq}{\begin{equation}}
\newcommand{\eeq}{\end{equation}}
\title{Existence and stability of periodic planar standing waves in phase-transitional
elasticity with strain-gradient effects}

\author{\sc \small
Jinghua Yao\thanks{Department of Mathematics,
Indiana University Bloomington, IN 47405, USA. Email:
yaoj@indiana.edu
 }}

\date{}

\begin{document}

\maketitle

\begin{abstract}
Extending investigations of Antman \& Malek-Madani, Schecter \&
Shearer, Slemrod, Barker \& Lewicka \& Zumbrun, and others, we
investigate phase-transitional elasticity models of strain-gradient
effect. We prove the existence of non-constant planar periodic
standing waves in these models with strain-gradient effects by
variational methods and phase-plane analysis, for deformations of
arbitrary dimension and general, physical, viscosity and
strain-gradient terms. Previous investigations considered
one-dimensional phenomenological models with artificial
viscosity/strain gradient effect, for which the existence reduces to
a standard (scalar) nonlinear oscillator. For our variational
analysis, we require that the mean vector of the unknowns over one
period be in the elliptic region with respect to the corresponding
pure inviscid elastic model. For our ($1$-D) phase-plane analysis,
we have no such restriction, obtaining essentially complete
information on the existence of non-constant periodic waves and
bounding homoclinic/heteroclinic waves. Our variational framework
has implications also for time-evolutionary stability, through the
link between the action functional for the traveling-wave ODE and
the relative mechanical energy for the time-evolutionary system.
Finally, we show that spectral implies modulational nonlinear
stability by using a change of variables introduced by Kotschote to
transform our system to a strictly parabolic system to which general
results of Johnson--Zumbrun apply.  Previous such results were
confined to one-dimensional deformations in models with artificial
viscosity--strain-gradient coefficients.
\end{abstract}

\begin{center}
\end{center}

{\bf 2010 MR Subject Classification}: 35Q74, 49A10, 49A22

\section{Introduction}

Elasticity is the typical property of elastic media. Since they have
a wide range of applications, the mathematical study of elasticity
has been an important topic (see
\cite{AB,AM,BLeZ,D,FP,K,RZ,S1,S2,S3,SS,Z1,Z2}, etc., and references
therein). However, up to now, to the best of our knowledge, the
study of phase-transitional elasticity has been carried out only for
phenomenological models \cite{S1,S2,S3,SS,Z2} for one-dimensional
shear flow, with classical double well potential and artificial
viscosity--capillarity terms. The treatment of the general,
physical, case was cited in \cite{BLeZ}, Appendix A, as an important
direction for further study.

In this paper, continuing the work of Antman and Malek-Madani
\cite{AM}, Slemrod \cite{S1,S2,S3}, Schecter and Shearer \cite{SS},
and Barker, Lewicka and Zumbrun \cite{BLeZ}, we study the existence
of planar elastic periodic traveling waves, compressible or
incompressible, for deformations of arbitrary dimensions, starting
from the most general form of the physical equations. It turns out
for general elasticity model with strain-gradient effects that,
similarly as observed for the phenomenological models studied
previously, the periodic traveling waves can only be standing waves
and the corresponding ordinary differential equation (ODE) system
exhibits Hamiltonian structure.

As we know, for a planar Hamiltonian system, we can use phase-plane
analysis to study its closed orbits. This corresponds to the case of
one-dimensional deformations, for which the unknowns are scalar in
the ODE system \eqref{twode} that we obtain. For higher dimensional
Hamiltonian systems, this method does not apply directly. In order
to prove the existence of non-constant periodic waves when the
unknowns are vectors, we consider the problem under the framework of
calculus of variation. There are several problems to overcome.
First, we need to formulate the problem in proper Banach spaces. It
turns out that the proper space for our purpose is the periodic
Sobolev space with mean zero property. Working in this framework
amounts to prescribing the mean of the unknown over one periodic (no
real restriction, since each periodic wave has a mean as long as it
exists).

Second, we need
to make sure that the waves we find are not constant waves. We overcome
this issue by considering the equations satisfied by the difference
between the original unknown and its mean. This makes the $0$
element in our working space always a critical point, which helps to
eliminate the possibility that the solutions we find are trivial.

Third, in the global model of elasticity, we need to consider the
assumption $\tau_3>0$ (see Section 2 or 3). This kind of condition
usually leads to a variational inequality and is related to an
obstacle problem. Meanwhile, this inequality condition makes our admissible
set (to make the wave physically meaningful) not weakly closed.
However, the asymptotic behavior of the elastic potential will help
overcome the related problem. We note that in the pure elastic case
without strain-gradient effects and viscosity, this restriction on
$\tau_3$ imposes significant challenges in the mathematical analysis
(see the discussions in \cite{Ba, AN}).

Besides the Hamitlonian structure of the standing wave equations, we prove that for
the general physical model, there exist non-constant periodic waves no matter whether the
unknowns are scalar or not (see Theorem 6.13) under assumptions on the mean vector and
periodic of the wave. For some specific phase-transitional models, we given explicit conditions
under which the non-constant oscillatory waves exist (see Section 7). In particular,
for the one dimensional models, we use phase-plane analysis to get detailed information on
the wave phenomena (existence of periodic, homoclinic, heteroclinic waves, see Section 8).

We address also the issue of nonlinear stability. Specifically, using a coordinate transformation introduced by Kotschote [K], we show
that the class of systems studied here are effectively strictly parabolic, in the
sense that they may be transformed to an enlarged strictly parabolic system.
This is similar in spirit to, but in practice quite different from, the change of variables introduced by Schecter
and Shearer [SS] in the special case of one-dimensional artificial viscosity/strain-gradient models.
Stability may then be treated within the general framework established for strictly parabolic systems
in [HZ, JZ] for which stability is completely understood, with the result that, up to nondegeneracy conditions,
spectral, linearized, and nonlinear stability are all equivalent, hence may be readily checked by numerical
Evans function computations.  See Appendix B for further discussion.

Comparing our results with others (see the references of this
paper), the problems here is interesting enough even only from the
modeling point of view, without even finding any waves. In
\cite{AM}, the authors treated the shear flow without the
strain-gradient effect and with isotropic assumption preventing
phase transition model (see discussion in \cite{BLeZ}). Here we
consider models with strain-gradient effects and the materials are
anisotropic, which gives rich wave phenomena. Antmann, Slemrod and
others (see \cite{A,AM,S1,S2,S3} and references therein) have
previously studied phenomenological $1$D phase-transitional models
with double-well potentials. Here we justify those types of
qualitative models by direct derivation from the physical shear flow
model; see Section \ref{s:just}.

It is well known that there are many outstanding open problems in
elasticity ranging from elastostatics to dynamics (see \cite{Ba}).
Due to strain-gradient effect
 and
associated higher regularity, we can give a neat treatment of our
problem. It would be very interesting to explore by numerics cases
that do not fit the hypotheses here (which are sufficient but by not necessary in the non-scalar case) but nonetheless support
periodic waves and also to explore either numerically or
analytically the spectral stability of these waves.  We hope to
address these issues in a followup work \cite{BYZ}. Numerical study of existence and stability of shock waves (which, since not necessarily zero-speed, are more plentiful)
would be another interesting direction for future study.

\section{Elasticity models with strain-gradient effects}

In this section, we will follow the presentations of \cite{AN, Ba,
BLeZ, NPT}. Let $\Omega$ be the reference configuration which models
an elastic body with constant temperature and density. A typical
point in $\Omega$ will be denoted by $X$. We use
$\xi:\Omega\times\mathbb{R}_+\rightarrow\mathbb{R}^3$ to denote the
deformation(i.e., the deformed position of the material point $X$).
Consequently, the deformation gradient is given by
$F:=\nabla_{X}\xi$, which we regard as an element in
$\mathbb{R}^{3\times3}$.

Adopting the notations above, the equations of isothermal elasticity
with strain-gradient effect are given through the following balance
of linear momentum
\begin{equation}\label{2.1}
\xi_{tt}-\nabla_{X}\cdot\Big(DW(\nabla\xi)+\mathcal
{Z}(\nabla\xi,\nabla\xi_t)-\mathcal {E}(\nabla^2\xi)\Big)=0.
\end{equation}

We make the following physical constraint on the deformation
gradient (see \cite{Ba, BLeZ} and \cite{AN, NPT} for the physical
background), prohibiting local self-impingement of the material:
\be\label{Fnonzero} \det F>0. \ee In \eqref{2.1}, the operator
$\nabla_X\cdot$ stands for the divergence of an approximate field.
As in \cite{D, NPT}, for a matrix-valued vector field, we use the
convention that the divergence is taken row-wise. In what follows,
we shall also use the matrix norm $|F|=(tr(F^TF))^{1/2}$, which is
induced by the inner product: $F_1:F_2:=tr(F_1^TF_2).$

In view of the second law of thermodynamics (see \cite{Ba, PB}), the
Piola-Kirchhoff stress tensor
$DW:\mathbb{R}^{3\times3}\rightarrow\mathbb{R}^{3\times3}$  is
expressed as the derivative of an elastic energy density
$W:\mathbb{R}^{3\times3}\rightarrow\mathbb{\overline{R}}_+$.
Throughout the paper, we assume as in \cite{AN, Ba, NPT} the elastic
energy density function $W$ is frame-indifference. Let $SO(3)$ be
the group of proper rotations in $\mathbb{R}^3$. Then the
frame-indifference assumption can be formulated as
\begin{equation}\label{2.3}
W(RF)=W(F),\quad\forall F\in\mathbb{R}^{3\times3},\quad \forall R\in
SO(3).
\end{equation}

Also, the material consistency (to avoid interpenetration of matter,
\eqref{Fnonzero},
[AN, Ba]) requires the following important assumption:
\begin{equation}\label{2.4}
W(F)\rightarrow+\infty\quad\mbox{as} \det F\rightarrow 0.
\end{equation}

We emphasize that viscous stress tensor
$\mathcal{Z}:\mathbb{R}^{3\times3}\times\mathbb{R}^{3\times3}\rightarrow\mathbb{R}^{3\times3}$
depends on both the deformation gradient $F$ and the velocity
gradient $Q=F_t=\nabla\xi_t=\nabla v$, where $\tau=\xi_t$. From
physical point of view, the stress tensor $\mathcal{Z}$ should also
be compatible with principles of continuum mechanics (balance of
angular momentum, frame invariance, and the Claussius-Duhem
inequality etc). For the related mathematical descriptions and
corresponding stress forms see \cite{AN, Ba, BLeZ} and references
therein.

The strain-gradient effect $\mathcal {E}$ is given by
$$\mathcal{E}(\nabla^2\xi)=\nabla_X\cdot D\Psi(\nabla^2\xi)=\Big[\sum_{i=1}^{3}\frac{\partial}{\partial X_i}\Big(\frac{\partial}{\partial(\partial_{ij}\zeta^k)}\Psi(\nabla^2\xi) \Big) \Big]_{j,k:1...3}$$
for some convex density $\Psi:
\mathbb{R}^{3\times3\times3}\rightarrow\mathbb{R}$, compatible with
frame indifference.

The corresponding inviscid part of system \eqref{2.1}
\begin{equation}\label{2.5}
\xi_{tt}-\nabla_X\cdot\Big(DW(\nabla\xi)\Big)=0
\end{equation}
can be written as
\begin{equation}\label{2.6}
(F,\tau)_t+\sum_{i=1}^{3}\partial_{X_i}\Big(\tilde{G}_i(F,\tau)\Big)=0.
\end{equation}
Above, $(F,\tau):\Omega\rightarrow\mathbb{R}^{12}$ represents
conserved quantities, while
$\tilde{G}_i:\mathbb{R}^{12}\rightarrow\mathbb{R}^{12}$ are given by
$$-\tilde{G}_i (F,\tau)=\tau^1 e_i\oplus \tau^2 e_i\oplus \tau^3 e_i\oplus\Big[\frac{\partial}{\partial F_{ki}}W(F)\Big]_{k=1}^{3},\,i=1...3$$
are the fluxes, and $e_i$ denotes the $i$-th coordinate vector in
$\mathbb{R}^3$.

The convex density $\Psi$ contributes
to equation \eqref{2.1} the term
\begin{equation}\label{2.7}
\nabla_{X}\cdot\Big( \mathcal {E}(\nabla^2\xi)\Big)=
\nabla_{X}\cdot\{\nabla_{X}\cdot D\Psi(\nabla^{2}\xi)\}.
\end{equation}
In view of the orders of differentiation and convexity of $\Psi$, we
may assume that
$$\Psi\geq 0;\quad\Psi(0)=0;\quad D\Psi(0)=0;\quad \delta Id\leq D^{2}\Psi(\cdot)\leq M Id$$
where $\delta, M$ are two positive real numbers and $Id$ is an
element in the space $\mathcal
{L}(\mathbb{R}^{3\times3\times3};\,\mathbb{R}^{3\times3\times3})$.
The mapping relations (ignoring physical constraints) are
$$\Psi: \mathbb{R}^{3\times3\times3}\rightarrow \mathbb{R}_{+};$$
$$D\Psi: \mathbb{R}^{3\times3\times3}\rightarrow\mathbb{R}^{3\times3\times3};$$
$$D^{2}\Psi: \mathbb{R}^{3\times3\times3}\rightarrow\mathcal {L}(\mathbb{R}^{3\times3\times3};\,\mathbb{R}^{3\times3\times3})$$

When the operator $\nabla_{X}\cdot$ reduces to the operator
$\partial_{x}$ where $x$ is a one dimension variable, \eqref{2.7}
takes the form $\partial_{x}\{\partial_{x}
D\Psi(\partial^{2}_{x}\xi)\}$. If we identify $\xi_{x}$ as $\tau$,
then $\partial^{2}_{x}\xi=\tau_{x}$ and $(2.7)$ becomes
$\partial_{x}\{\partial_{x}
D\Psi(\tau_x)\}=\partial_{x}\{D^{2}\Psi(\tau_x)\tau_{xx}\}$. Note
that $D^{2}\Psi:\mathbb{R}^{3}\rightarrow\mathcal
{L}(\mathbb{R}^{3};\,\mathbb{R}^{3})$ when $\nabla_{X}\cdot$ reduces
to $\partial_{x}$. So we assume that $D^{2}\Psi(\cdot)$ as matrix
function satisfy the assumption $\delta Id\leq D^{2}\Psi(\cdot)\leq
M Id\quad\mbox{as operators}$.

\section{Equations and specific models}\label{s:specific}
In this paper, we focus on the interesting subclass of planar
solutions, which are solutions in the full 3D space that depend only
on a single coordinate direction; that is, we investigate
deformations $\xi$ given by
$$\xi(X)=X+U(z),\quad X=(x,y,z),\quad U=(U_1,U_2,U_3)\in\mathbb{R}^3.$$
Corresponding to the above deformation or displacement $\xi$, the
deformation gradient with respect to $X$
\begin{equation}\label{3.1}
F=\begin{pmatrix}1 & 0 & U_{1,z}\\ 0 & 1 & U_{2,z}\\0 & 0 &
1+U_{3,z}\end{pmatrix}=\begin{pmatrix}1 & 0 & \tau_1\\ 0 & 1 &
\tau_2\\0 & 0 & \tau_3\end{pmatrix}.
\end{equation}
We shall denote $V=(\tau,u)=(\tau_1,\tau_2,\tau_3,u_1,u_2,u_3)$,
where $\tau_1=U_{1,z},\tau_2=U_{2,z},\tau_3=1+U_{3,z}$ and
$u_1=U_{1,t}, u_2=U_{2,t}, u_3=U_{3,t}$ with the physical constraint
$\tau_3>0$, corresponding to $\det F>0$ in the region of physical
feasibility of $V$.

Writing $W(\tau)=W\Big(\begin{pmatrix}1 & 0 & \tau_1\\ 0 & 1 &
\tau_2\\0 & 0 & \tau_3\end{pmatrix}\Big)$, we see that for all $F$
as in \eqref{2.1} there holds
$$\nabla_X\cdot(DW(F))=(D_{\tau}W(\tau))_z.$$
That is, the planar equations inherit a vector-valued variational
structure echoing the matrix valued variational structure (note that the left hand side is the divergence of $DW(F)$). \\

In this paper, we first study the problems (traveling wave ODEs,
Hamiltonian ODEs, existence of standing waves) for general elastic
potential energy and give a rather general abstract existence
result. Then we
study local models by specifying the related terms in system \eqref{2.1}
as follows:

\medskip

 \textbf{1. Elastic potential W.} As described
 in [FP], we shall study the phase-transitional elastic potential:
 $$W(F)=|F^TF-C_-|^2\cdot|F^TF-C_+|^2,$$
which is a potential for anisotropic material with material frame
indifference property. We study models involving this
phase-transitional elasticity potential(See Appendix A for related
computations). It is important to notice that the elastic potential
here does not satisfy the asymptotic behavior when $\det
F\rightarrow 0^+$ . Especially this is an irrelevant assumption for
shear models. Hence the related models are local models for the real
physics.
\medskip

\textbf{2. Viscous stress tensor $\mathcal{Z}$.} We use the
following tress tensor which is compatible with the principles of
continuum mechanics (see \cite{BLeZ})
$$\mathcal{Z}(F,Q)=2(\det F)sym(QF^{-1})F^{-1,T}.$$
We note that the related Cauchy stress tensor $T_2=2(\det
F)^{-1}\mathcal{Z}F^T=2sym(QF^{-1})$ is the Lagrangian version of
the stress tensor $2sym\nabla v$ written in the Eulerian
coordinates. For incompressible fluids $2div(sym\nabla v)$, giving
the usual parabolic viscous regularization of the fluid dynamics
evolutionary system.
\medskip

\textbf{3. The Strain-gradient term $\mathcal{E}$.} For the
strain-gradient effect we will choose $\Psi(P)=\frac{1}{2}|P|^2$, so
that
$\mathcal{E}(\nabla^2\xi)=\nabla_X\cdot\nabla^2\xi=\triangle_XF,$
which is an extension of the 1D case of [S1]. We see that it is the strain-gradient term that makes the models have abundant wave phenomena.\\

\textbf{The system.} As a convention, we shall use
$x\in\mathbb{R}^1$ as the space variable instead of $z$. So we have
the following system
\begin{equation}\label{3.2}
\begin{cases}
\tau_t-u_x=0;\\
u_t +\sigma(\tau)_x=(b(\tau)u_x)_x - (d(\tau_x)
\tau_{xx})_x.\end{cases}
\end{equation}
with
 \be\label{b3_1}
\hbox{\rm $\sigma:=-D_{\tau} W(\tau)$, $d(\cdot):=D^2\Psi(\cdot)=Id$
and $b(\tau)=\tau_3^{-1}\bp 1&0&0\\0&1&0\\0&0&2\ep$. } \ee

We are interested in the existence of periodic traveling waves of
the above system, which involves third order term because of the
strain-gradient effect. See also Appendix \ref{s:structure} for some
structure properties of the system.

\section{Traveling Wave ODE System}

We seek traveling wave solution of the system $(3.2)$,
$(\tau(x,t),u(x,t)):=(\tau(x-st), u(x-st))$, where $s\in\mathbb{R}$
is the wave speed. Let us denote in the following $'$ as
differentiation with respect to $x-st$. For convenience, we still
use $x$ to represent $x-st$ (Indeed, we will show a bit later that
in fact $s=0$ is necessary for the existence of periodic or
homoclinic waves; see equation (5.6)). With further investigation in
mind, we write the related equations for the general class of
elastic models with strain-gradient effects. Now from system
$(3.2)$, we have the ODE system \be
\begin{cases} -s\tau'-u'=0;\\
-su'+\sigma(\tau)'=(b(\tau)u')'-(d(\tau')\tau'')'.
\end{cases}
\ee

Plugging the first equation into the second in the above system, we
obtain the following second-order ODE in $\tau$: \be
s^2\tau'+\sigma(\tau)'=-(b(\tau)s\tau')'-(d(\tau')\tau'')'. \ee

In view of $d(\cdot)=D^2\Psi(\cdot)$, we readily see:
\be\label{4.2}
s^2\tau'+\sigma(\tau)'=-(b(\tau)s\tau')'-(D^2\Psi(\tau')\tau'')' .
\ee
Choosing a specific space point, say $x_0$, we integrate once to
get:
\be \label{4.3}
s^2\tau+\sigma(\tau)+q=-sb(\tau)\tau'-D\Psi(\tau')'
\ee
Here
$q$ is an integral constant vector. Relating this with the elastic
potential function $W$, we have
\be\label{twode}
-DW(\tau)+s^2\tau+q=-sb(\tau)\tau'-D\Psi(\tau')'
\ee

Note carefully that the integral constant vector is given by
\be\label{q}
q=\{DW(\tau)-s^2\tau-sb(\tau)\tau'-D\Psi(\tau')'\}\Big|_{x=x_0}.\ee

\section{Hamiltonian Structure}\label{s:ham}
Defining $G(P):=\langle P,D\Psi(P)\rangle-\Psi(P)$, we see that
$\frac{dG}{dP}=\langle P,D^2\Psi\rangle$. Here $P\in\mathbb{R}^n$
and $\Psi:\mathbb{R}^n\rightarrow\mathbb{R}$(for our purpose
$n=1,2,3$), $G:\mathbb{R}^n\rightarrow\mathbb{R}$ a different scalar
potential type function. Now we are ready to state a structural
property about the traveling wave ODE system \eqref{twode}.

\begin{proposition} When $s=0$, the system \eqref{twode}
is a Hamiltonian system with factor $\Big(D^2\Psi(\tau')\Big)^{-1}$,
preserving the Hamiltonian integral
$$H(\tau,\tau')=-W(\tau)+q\tau+G(\tau')\equiv\mbox{constant}.$$
Here $W$ can be taken in particular as the phase-transitional
elastic potential (see Appendix \ref{s:phase_app}
and also p. 36 of \cite{BLeZ}),
with $(a_1,a_2,a_3)=(\tau_1,\tau_2,\tau_3)$)
$$
W(\tau)=\Big(2(\tau_1^2 + \tau_2^2 + \eps^2) +( |\tau|^2-1
-\eps^2)^2 \Big)^2 - 16\eps^2 \tau_2^2,
$$
and $\Psi(p)$ as any convex function with $\Psi(0)=0$ and
$d\Psi(0)=0$, simplest case $\Psi(P)=|P|^2/2$. (Note: in this simple
case $G(P)=\Psi(P)$.)
\end{proposition}

\begin{proof}
When $s=0$, the traveling wave ODE \eqref{twode} becomes:
\begin{equation}\label{5.1}
-dW(\tau)+q=-D\Psi(\tau')'
\end{equation}
and the constant $q=\{DW(\tau)-D\Psi(\tau')'\}\Big |_{x=x_0}$.
Noticing the positive-definiteness of $D^{2}\Psi(\cdot)$, we may
write the ODE as a first order system by regarding $\tau,\tau'$ as
independent variables:

\ba\label{visceq}
\tau'&=\tau'=[D^{2}\Psi(\tau')]^{-1}D^{2}\Psi(\tau')\tau';\\
\tau''&=-[D^{2}\Psi(\tau')]^{-1}(-DW(\tau)+q)\\
\ea

Now, consider the energy surface given by: \be\label{H*} H(\tau,
\tau'):= -W(\tau) + q\tau+ G(\tau')  . \ee We see that
\ba\label{Hvisceq}
\frac{\d}{\d\tau'}H(\tau,\tau')&=\frac{dG(\tau')}{d\tau'}=D^2\Psi(\tau')\tau';\\
\frac{\d}{\d\tau}H(\tau,\tau')&=-DW(\tau)+q.\\
\ea

Comparing \eqref{visceq}, \eqref{Hvisceq},  we see that the
traveling wave ODE is a Hamiltonian system with factor
$\gamma:=[D^{2}\Psi(\tau')]^{-1}$. Thus, \eqref{twode} preserves the
Hamiltonian $H$. We can see this also by the explicit computation,
writing $\zeta=x-st$:
\begin{align*}
\frac{d}{d\zeta}H(\tau,\tau')=&\frac{\d}{\d\tau}H(\tau,\tau')\tau' +
\frac{\d}{\d\tau'}H(\tau,\tau')\tau''\\
=&\gamma\frac{\d}{\d\tau}H(\tau,\tau')\frac{\d}{\d\tau'}H(\tau,\tau')
+
\gamma\frac{\d}{\d\tau'}H(\tau,\tau')\{-\frac{\d}{\d\tau}H(\tau,\tau')\}
%\\=&
=
0.
\end{align*}
\end{proof}

From the above structural information, we easily get a necessary
condition for the existence of periodic or homoclinic waves,
extending results of \cite{OZ} in a one-dimensional model case.\\

\begin{theorem}
For \eqref{twode} with $s\gtrless 0$,
there holds $dH/d\zeta\lessgtr
0$, where
\be\label{H}
H(\tau, \tau'):= -W(\tau) +
\frac{s^2}{2}|\tau|^2 +q\tau+ G(\tau') ,
\ee
so that no homoclinic
or periodic orbits can occur unless $s=0$.
\end{theorem}

\begin{proof}
Considering the evolution of
$\frac{d}{d\zeta}H(\tau,\tau')$ along the flow of traveling wave ODE
system \eqref{twode}, we have
\begin{align*}
\frac{d}{d\zeta}H(\tau,\tau')=&\frac{\d}{\d\tau}H(\tau,\tau')\tau'+\frac{\d}{\d\tau'}H(\tau,\tau')\tau''\\
=&\langle -D_{\tau}W(\tau)+q+s^{2}\tau, \tau' \rangle+\langle
DG(\tau'),
\tau''\rangle\\
=&\langle -D_{\tau}W(\tau)+q+s^{2}\tau, \tau' \rangle+\langle
D^{2}\Psi(\tau')\tau', \tau''\rangle\\
=&\langle -D_{\tau}W(\tau)+q+s^{2}\tau, \tau' \rangle+\langle
D^{2}\Psi(\tau')\tau'', \tau'\rangle\\
=&\langle -D_{\tau}W(\tau)+q+s^{2}\tau+D\Psi(\tau')', \tau' \rangle\\
=&\langle -sb(\tau)\tau', \tau'\rangle.
\end{align*}
The conclusion thus follows from the positive definiteness of
$b(\tau)$.
\end{proof}

{\bf The Hamiltonian system.} From the above analysis, we see that
necessarily $s=0$, i.e., all traveling periodic waves are standing.
The traveling wave ODE system reduces to the following form with an
integral constant $q$

\begin{equation}\label{5.6}
\begin{cases}-\tau''=-D_{\tau}W(\tau)+q;\\
q=\{D_{\tau}W(\tau)-\tau''\}\Big|_{x=x_0}.\end{cases}
\end{equation}

If we take the Hamiltonian point of view, the corresponding
Hamiltonian for the above system is
$$H(\tau,\tau')=\frac{1}{2}|\tau'(x)|^2+V(\tau,\tau'),$$
where $V(\tau,\tau'):=q\cdot\tau(x)-W(\tau(x))$. The periodic
solutions of the system are confined to the surface
$H(\tau,\tau')\equiv\mbox{constant}$.

In the following, we list the elastic potential and related
information for the phase-transitional models we shall deal with in
this paper for completeness and future study. To get these models,
we fix one or two directions of $\tau$ as zero, or, in the
incompressible case, $\tau_3\equiv 1$ (as described in
\cite{AM,BLeZ}, the latter is an imposed constraint, that is
compensated for in the $\tau_3$ equation by a Lagrange multiplier
corresponding to pressure). We refer the reader to
\cite{BLeZ}, Section 3, for details of the derivations of these models.\\

\subsection{2D Incompressible Shear Model} This model corresponds to
setting $\tau_3=1$.
\begin{equation}\label{5.7}
W(\tau)=\Big(2\tau_1^2 + 2(\tau_2-\eps)^2 +(|\tau|^2-\eps^2)^2 \Big)
\Big(2\tau_1^2 + 2(\tau_2+\eps)^2 +(|\tau|^2-\eps^2)^2 \Big).
\end{equation}
Its gradient components are
\begin{equation}\label{5.8}
D_{\tau_1}W(\tau)=8\tau_{1}(|\tau|^{2}+1-\varepsilon^{2})\{2(|\tau|^{2}+\varepsilon^2)+(|\tau|^{2}-\varepsilon^{2})^{2}\};
\end{equation}
\begin{equation}\label{5.9}
D_{\tau_2}W(\tau)=8\tau_{2}(|\tau|^{2}+1-\varepsilon^{2})\{2(|\tau|^{2}+\varepsilon^2)+(|\tau|^{2}-\varepsilon^{2})^{2}\}-32\tau_{2}\varepsilon^2.
\end{equation}
The Hessian components are
\begin{equation}\label{5.10}
w_{11}:=D_{\tau_1\tau_1}W(\tau)=8(|\tau|^{2}+1-\varepsilon^{2}+2\tau_1^2)\{2(|\tau|^{2}+\varepsilon^2)+(|\tau|^{2}-\varepsilon^{2})^{2}\}+32\tau_{1}^{2}(|\tau|^2+1-\varepsilon^2)^2;
\end{equation}
\begin{equation}\label{5.11}
w_{12}=w_{21}:=D_{\tau_1\tau_2}W(\tau)=16\tau_{1}\tau_{2}\{2(|\tau|^{2}+\varepsilon^2)+(|\tau|^{2}-\varepsilon^{2})^{2}\}
+32\tau_{1}\tau_{2}(|\tau|^{2}+1-\varepsilon^2)^2;
\end{equation}
\begin{equation}\label{5.12}
w_{22}:=D_{\tau_2\tau_2}W(\tau)=8(|\tau|^{2}+1-\varepsilon^2+2\tau_{2}^2)[2(|\tau|^{2}+\varepsilon^2)+(|\tau|^{2}-\varepsilon^2)^2]
+32[\tau_{2}^{2}(|\tau|^{2}+1-\varepsilon^2)^2-\varepsilon^2].
\end{equation}

\subsection{1D Shear Model I: $\tau_3\equiv1$; $\tau_2\equiv0$.} The
elastic potential becomes
\begin{equation}\label{5.13}
\begin{aligned}
W(\tau)&= \Big(2\tau_1^2 + 2\eps^2 +(\tau_1^2-\eps^2)^2 \Big)^2.
\end{aligned}
\end{equation}
The first order derivative is
\begin{equation}\label{5.14}
D_{\tau_1}W(\tau)=8\tau_{1}(\tau_1^2+1-\varepsilon^2)\{2(\tau_1^{2}+\varepsilon^2)+(\tau_1^{2}-\varepsilon^{2})^{2}\}.
\end{equation}
The second order derivative is
\begin{equation}\label{5.15}
D_{\tau_1\tau_1}W(\tau)=8(3\tau_1^{2}+1-\varepsilon^{2})\{2(\tau_1^{2}+\varepsilon^2)+(\tau_1^{2}-\varepsilon^{2})^{2}\}+32\tau_{1}^{2}(\tau_1^2+1-\varepsilon^2)^2.
\end{equation}
\subsection{1D Shear Model II: $\tau_3\equiv1;\tau_1\equiv0$.}
Correspondingly, the elastic potential becomes
\begin{equation}\label{5.16}
\begin{aligned}
W(\tau)&= \Big(2(\tau_2-\varepsilon)^2+(\tau_2^2-\eps^2)^2
\Big)\times \Big(2(\tau_2+\varepsilon)^2+(\tau_2^2-\eps^2)^2\Big)
\end{aligned}
\end{equation}
The first order derivative is
\begin{equation}\label{5.17}
D_{\tau_2}W(\tau)=8\tau_{2}(\tau_2^2+1-\varepsilon^2)\{2(\tau_2^{2}+\varepsilon^2)+(\tau_2^{2}-\varepsilon^{2})^{2}\}-32\tau_2\varepsilon^2;
\end{equation}
The second order derivative is
\begin{equation}\label{5.18}
D_{\tau_2\tau_2}W(\tau)=8(3\tau_2^{2}+1-\varepsilon^{2})\{2(\tau_2^{2}+\varepsilon^2)+(\tau_2^{2}-\varepsilon^{2})^{2}\}+32\{\tau_{2}^{2}(\tau_2^2+1-\varepsilon^2)^2-\varepsilon^2\}.
\end{equation}

\br[$1$D vs. $2$D shear solutions]\label{realrmk}
Evidently, solutions of $1$-D shear models I or II determine solutions of
the full $2$
 shear model obtained by adjoining
$ \tau_2\equiv 0$ or $\tau_1\equiv 0$ respectively. It is worth
noting that the structure $dW(\tau)=\tau f(|\tau|)+
%\tau\otimes \tau+
c(0,\tau_2)^T$
for $f$ a scalar-valued function and $c$ a scalar constant yields
that the only solutions $\tau(x)$
of \eqref{5.6} with $\eta\cdot \tau\equiv c_2=\const$
for some constant vector $\eta$ are those satisfying
$c_2 f(|\tau|)+c\eta_2\tau_2 \equiv \const$, which gives by direct
computation $\tau\equiv \const$
or else $\eta \cdot \tau\equiv 0$ and $\eta_2\tau_2\equiv 0$,
in which case $\tau_1\equiv 0$ or $\tau_2\equiv 0$.
That is, the $1$D systems derived here are the {\rm only}
solutions of the $2$D shear model that are not genuinely two-dimensional
in the sense that they are confined to a line in the $\tau$-plane.
In particular, if the mean of $\tau_1$ or $\tau_2$ over one period
is not zero, then we can be sure that the solution is genuinely two-dimensional.
\er

\subsection{1D Compressible Model III} In this case $\tau_1=\tau_2\equiv 0$
and we denote $\tau=\tau_3$.
The potential and its derivatives are given below.
The elastic potential becomes
\begin{equation}\label{5.19}
\begin{aligned}
W(\tau)&= \Big(2\eps^2 +(\tau_3^2-1-\eps^2)^2 \Big)^2
\end{aligned}
\end{equation}
\begin{figure}

\end{figure}
 The first order derivative is
\begin{equation}\label{5.20}
D_{\tau_3}W(\tau)=8\tau_{3}(\tau_3^{2}-1-\varepsilon^{2})\{2\varepsilon^2+(\tau_3^{2}-1-\varepsilon^{2})^{2}\};
\end{equation}
And the second order derivative is
\begin{equation}\label{5.21}
w_{33}:=D_{\tau_3\tau_3}W(\tau)=8(3\tau_3^2-1-\varepsilon^2)\{2\varepsilon^2+(\tau_3^{2}-1-\varepsilon^{2})^{2}\}
+32\tau_{3}^2(\tau_3^{2}-1-\varepsilon^2)^2.
\end{equation}

\subsection{2D compressible models.}

First, we consider the case
$\tau=(\tau_2,\tau_3)^T\in\mathbb{R}^2_+$. The elastic potential $W$
and derivatives are as follows.
\begin{equation}\label{5.22}
W(\tau)= \Big(2(\tau_2-\eps)^2 +(|\tau|^2-1-\eps^2)^2
\Big)\Big(2(\tau_2+\eps)^2 +(|\tau|^2-1-\eps^2)^2 \Big);
\end{equation}
The gradient components are
\begin{equation}\label{5.23}
D_{\tau_2}W(\tau)=8\tau_{2}(|\tau|^{2}-\varepsilon^{2})\{2(\tau_2^{2}+\varepsilon^2)+(|\tau|^{2}-1-\varepsilon^{2})^{2}\}-32\tau_{2}\varepsilon^2.
\end{equation}
\begin{equation}\label{5.24}
D_{\tau_3}W(\tau)=8\tau_{3}(|\tau|^{2}-1-\varepsilon^{2})\{2(\tau_2^{2}+\varepsilon^2)+(|\tau|^{2}-1-\varepsilon^{2})^{2}\};
\end{equation}
Similarly, we have the Hessian components
\begin{equation}\label{5.25}
w_{22}:=D_{\tau_2\tau_2}W(\tau)=8(|\tau|^{2}-\varepsilon^2+2\tau_{2}^2)[2(\tau_2^{2}+\varepsilon^2)+(|\tau|^{2}-1-\varepsilon^2)^2]
+32[\tau_{2}^{2}(|\tau|^{2}-\varepsilon^2)^2-\varepsilon^2].
\end{equation}
\begin{equation}\label{5.26}
w_{23}=w_{32}:=D_{\tau_3\tau_2}W(\tau)=16\tau_{2}\tau_{3}\{2(\tau_2^{2}+\varepsilon^2)+(|\tau|^{2}-1-\varepsilon^{2})^{2}\}
+32\tau_{2}\tau_{3}(|\tau|^{2}-\varepsilon^2)(|\tau|^{2}-1-\varepsilon^2)
\end{equation}
\begin{equation}\label{5.27}
w_{33}:=D_{\tau_3\tau_3}W(\tau)=8(|\tau|^2-1-\varepsilon^2+2\tau_3^2)\{2(\tau_2^{2}+\varepsilon^2)+(|\tau|^{2}-1-\varepsilon^{2})^{2}\}
+32\tau_{3}^2(|\tau|^{2}-1-\varepsilon^2)^2
\end{equation}

Second, if we fix the $\tau_2$ direction and let
$\tau:=(\tau_1,\tau_3)^T$, we get another 2D compressible model. We
omit the details here as the form is obvious.

\subsection{The full 3D model.} In this case $\tau=(\tau_1,\tau_2,\tau_3)^T\in\mathbb{R}^3_+$. corresponding to the
phase-transitional elastic potential function $W$, we list the
components of $D^2W(\tau):=(w_{ij})_{3\times3}$.
\begin{equation}\label{5.28}
w_{11}=8(|\tau|^{2}+2\tau_{1}^{2}-\varepsilon^2)\{2(|\tau|^{2}-\tau_{3}^{2}+\varepsilon^2)+(|\tau|^{2}-1-\varepsilon^{2})^{2}\}
+32\tau_{1}^{2}(|\tau|^{2}-\varepsilon^2)^2
\end{equation}

\begin{equation}\label{5.29}
w_{12}=w_{21}=16\tau_{1}\tau_{2}\{2(|\tau|^{2}-\tau_{3}^{2}+\varepsilon^2)+(|\tau|^{2}-1-\varepsilon^{2})^{2}\}
+32\tau_{1}\tau_{2}(|\tau|^{2}-\varepsilon^2)^2
\end{equation}

\begin{equation}\label{5.30}
w_{13}=w_{31}=16\tau_{1}\tau_{3}\{2(|\tau|^{2}-\tau_{3}^{2}+\varepsilon^2)+(|\tau|^{2}-1-\varepsilon^{2})^{2}\}
+32\tau_{1}\tau_{3}(|\tau|^{2}-\varepsilon^2)(|\tau|^{2}-1-\varepsilon^2)
\end{equation}

\begin{equation}\label{5.31}
w_{23}=w_{32}=16\tau_{2}\tau_{3}\{2(|\tau|^{2}-\tau_{3}^{2}+\varepsilon^2)+(|\tau|^{2}-1-\varepsilon^{2})^{2}\}
+32\tau_{2}\tau_{3}(|\tau|^{2}-\varepsilon^2)(|\tau|^{2}-1-\varepsilon^2)
\end{equation}

\begin{equation}\label{5.32}
w_{33}=8(|\tau|^{2}+2\tau_{3}^{2}-1-\varepsilon^2)\{2(|\tau|^{2}-\tau_{3}^{2}+\varepsilon^2)+(|\tau|^{2}-1-\varepsilon^{2})^{2}\}
+32\tau_{3}^{2}(|\tau|^{2}-1-\varepsilon^2)^{2}
\end{equation}

\begin{equation}\label{5.33}
w_{22}=8(|\tau|^{2}-\varepsilon^2+2\tau_{2}^2)[2(\tau_{1}^{2}+\tau_{2}^{2}+\varepsilon^2)+(|\tau|^{2}-1-\varepsilon^2)^2]
+32[\tau_{2}^{2}(|\tau|^{2}-\varepsilon^2)^2-\varepsilon^2].
\end{equation}

\br\label{realrmk2}
Similarly as in Remark \ref{realrmk}, we find that solutions of the
$2$D compressible models are genuinely two-dimensional, in the sense
that they are not confined to a line in the $\tau$-plane, unless
they are solutions of the $1$D model derived above: in particular,
the mean over one period of $(\tau_1,\tau_2)$ is zero.
Likewise, solutions of the full $3$D compressible model are genuinely
three-dimensional in the sense that they are not confined to a plane,
unless they are solutions of one of the $2$D models derived above, in
particular, the mean of $\tau_1$ or of $\tau_2$ over one period is zero.
\er

\subsection{Justification of phenomenological models}\label{s:just}
We note that, for the case of $1$-D shear flow, the coefficients
given by \eqref{b3_1} become $b,d\equiv \const$, and the elastic
potential $W$ is of a generalized double-well form. Thus, we recover
from first principles the type of phenomenological model studied in
\cite{S1,S2,S3,SS,Z2}, though with a slightly modified potential
refining the quartic double-well potential assumed in the
phenomenological models. The $2$-D shear flow gives a natural
extension to multi-dimensional deformations, which is also
interesting from the pure Calculus of Variations point of view (see
the following section), as a physically relevant example of a
vectorial ``real Ginzberg--Landau'' problem of the type studied on
abstract grounds by many authors. Finally, we note that the various
compressible models give a different extension of the
phenomenological models, to the case of ``real'' or nonconstant
viscosity.

\section{Calculus of Variations.}
In this section, we formulate the problem in the framework of
Calculus of Variations and give the proof of the existence result.
\subsection{Space structure}
As a first step, we recall the notions of Sobolev spaces involving
periodicity and introduce the space structure we are going to use
(see \cite{MW}). For fixed real number $T>0$, let $C^{\infty}_T$ be
the space of infinitely differentiable $T$-periodic functions from
$\mathbb{R}$ to
$\mathbb{R}^n$ (for our purpose $n=1,2,3$.).\\

\begin{lem} Let $u,v\in L^1 (0,T;\mathbb{R}^n)$. If the following holds:
for every $f\in C^{\infty}_T$,
$$\int_{0}^T (u(t), f'(t))dt=-\int_0^T (v(t), f(t))dt,$$
then
$$\int_0^T v(s)ds=0$$
and there exists a constant vector $c$ in $\mathbb{R}^N$ such that
$$u(t)=\int_0^t v(s)ds+c\quad\mbox{a.e. on} [0, T].$$
The function $v:=u'$ is called the \textbf{weak derivative of u}.
Consequently, we have
$$u(t)=\int_0^t u'(l)dl+c,$$
which implies the following:
$$u(0)=u(T)=c;$$
$$u(t)=u(s)+\int_s^t u'(l)dl.$$
\end{lem}

\begin{proof}
For the mean zero property, we could consider the
specific test function $f=e_j$. For the integral formulation, we could consider
the use of Fubini Theorem and Fourier expansion of $f$ to
conclude (\cite{MW}).
\end{proof}

Define the Hilbert space $H^1_T$ as usual (hence reflexive Banach
space) with this inner product and corresponding norm: for $u,v\in
H^1_T$,
$$\langle u,v\rangle:=\int_0^T (u, v)+(u', v')ds;$$
$$\|u\|^2:=\int_0^T|u|^2+|u'|^2 ds.$$\\

Next, we collect some facts for later use.

\begin{proposition} (Compact Sobolev Embedding property)
$H^1_T\subset\subset C[0,T]\mbox{compactly}$.
\end{proposition}

\begin{proposition} If $u\in H^1_T$ and
$(1/T)\int_0^Tu(t)\,dt=0$, then we have Wirtinger's inequality
$$\int_0^T|u(t)|^2\,dt\leq(T^2/4\pi^2)\int_0^T|u'(t)|^2\,dt$$
and a Sobolev inequality
$$|u|^2_{\infty}\leq (T/12)\int_0^T|u'(t)|^2\,dt.$$
\end{proposition}

The Compact Sobolev imbedding property will give us the required
weak lower semi-continuity property for the nonlinear functionals.
The Wirtinger's inequality supplies us equivalent norms in related
Sobolev spaces with mean zero property (see \cite{MW} for complete
proofs).

\subsection{Variational formulation of the problems}

Now for a given real positive constant $T$, we consider problem
\eqref{5.6} in $H^1_T$
\begin{equation*}
\begin{cases}
-\tau''=-D_{\tau}W(\tau)+q=-D_{\tau}(W(\tau)-q\cdot\tau);\\
\tau(0)-\tau(T)=0;\tau'(0)-\tau'(T)=0.
\end{cases}
\end{equation*}

Let us first consider the cases and formulations without the
physical restriction $\tau_3>0$. Assume that:
$$\bar\tau:=\frac{1}{T}\int_0^T \tau(x)dx=m.$$
Here $m\in \mathbb{R}^n$, $n=1,2,3$ and we will use bar to represent
mean over one period similarly. Hence, we consider the following
problem

\begin{equation}\label{6.1}
\begin{cases}
\tau''(x)=DW(\tau)-q;\\
\tau(0)=\tau(T);\tau'(0)=\tau'(T);\\
\frac{1}{T}\int_0^T \tau(x)\,dx=m.
\end{cases}
\end{equation}

If we seek periodic solutions, $q$ can be determined by integrating
the equations above over one period; that is,
$$q=\frac{1}{T}\int_0^T DW(\tau(x))\,dx.$$

Define $v(x)=\tau(x)-m$. We see easily that $\frac{1}{T}\int_0^T
v(x)\,dx=0,$ and $v(x)$ satisfies the system of equations:

\begin{equation}\label{6.2}
\begin{cases}
v''(x)=DW(v+m)-q;\\
v(0)=v(T);v'(0)=v'(T);\\
\frac{1}{T}\int_0^T v(x)\,dx=0.
\end{cases}
\end{equation}

For convenience, we rewrite the above system as
\begin{equation}\label{6.3}
\begin{cases}
v''(x)=DW(v+m)-DW(m)+DW(m)-q;\\
v(0)=v(T);v'(0)=v'(T);\\
\frac{1}{T}\int_0^T v(x)\,dx=0.
\end{cases}
\end{equation}

Define $\tilde{W}(v)=W(v+m)-DW(m)\cdot v$ and $\tilde{q}=q-DW(m)$.
We get the following problem

\begin{equation}\label{6.4}
\begin{cases}
v''(x)=D\tilde{W}(v)-\tilde{q};\\
v(0)=v(T);v'(0)=v'(T);\\
\frac{1}{T}\int_0^T v(x)\,dx=0.
\end{cases}
\end{equation}
Here $\tilde{q}$ is determined by integration: $\frac{1}{T}\int_0^T
D\tilde{W}(v)\,dx=\tilde{q}$.

\begin{remark}\label{rmk6.3}
We require $v_3>-m_3$ on $[0, T]$ for models involving $\tau_3$
direction in view of the physical assumption $(2.2)$.
\end{remark}

Define $F(v)=W(v+m)-W(m)-DW(m)\cdot v$ and introduce the functional

\begin{equation}
\cI(v)=\int_0^T\frac{1}{2}|v'|^2\,dx+\int_0^T F(v)\,dx
\end{equation}
on the space
$$H^1_{T,0}:=\{v\in H^1_T; \bar{v}=\frac{1}{T}\int_0^T v\,dx=0\}.$$

\begin{proposition} $0$ is always a critical point of the functional $\cI$ defined
above on $H^1_{T,0}$.
\end{proposition}

\begin{proof}
It is easy to verify that for $\phi\in H^1_{T,0}$,
there holds
$$I'(v)\phi=\int_0^T v'\cdot\phi'+D\tilde{W}(v)\cdot\phi\,dx.$$
Taking $v=0$ and noticing that $D\tilde{W}(0)=0$, we get the desired
result.
\end{proof}

\begin{remark} For any Hamiltonian ODE, there exists such an equivalent variational, or ``Lagrangian'' formulation,
according to the principle of least action; see \cite{Lan}. By our
formulation, we make $0$ always a critical point and it corresponds
to the constant solution. This geometric property supplies us nice
way to exclude the possibility that the periodic solution we find is
constant, i.e., to help prove the periodic waves we find are
oscillatory.
\end{remark}

\begin{proposition} Without physical restriction on $\tau_3$, the critical point
of $\cI$ corresponds to the solution of $(5.6)$.
\end{proposition}

\begin{proof}
This can be regarded as a simple consequence of Corollary 1.1 in
\cite{MW}. For completeness, we write the details here. First,
assume $v$ solves
\begin{equation}\label{6.5}
\begin{cases}
v''(x)=D\tilde{W}(v)-\tilde{q};\\
v(0)=v(T);v'(0)=v'(T);\\
\frac{1}{T}\int_0^T v(x)\,dx=0.
\end{cases}
\end{equation}

Multiplying the equation by $\phi\in H^1_{T,0}$ and integrating to
get
$$\int_0^T v'\phi'+D\tilde W(v)\cdot\phi\,dx=0,$$
i.e $v$ is a critical point of $\cI$.

Next, we assume $v$ is a critical point and $\phi\in H^1_{T}$. Then
$\phi-\bar\phi\in H^1_{T,0}$. Hence we have
$$\int_0^T v'\cdot(\phi-\bar\phi)'+\int_0^T D\tilde W(v)\cdot(\phi-\bar\phi)\,dx=0$$
i.e.
$$\int_0^T v'\cdot\phi'+\int_0^T D\tilde W(v)\cdot\phi-\int_0^T D\tilde W(v)\cdot\bar\phi\,dx=0$$
Noting that $\bar\phi=\frac{1}{T}\int_0^T \phi\,dx$,
we find that the left-hand
side expression above is:
$$\int_0^T v'\cdot\phi'+\int_0^T D\tilde W(v)\cdot\phi-\int_0^T D\tilde W(v)\cdot(\frac{1}{T}\int_0^T \phi\,dx)\,dx=0$$

Noticing that $\frac{1}{T}\int_0^T D\tilde{W}(v)\,dx=\tilde{q}$, we
get

$$\int_0^T v'\cdot\phi'+\int_0^T (D\tilde W(v)-\tilde{q})\cdot\phi\,dx=0,$$
which implies $v''=D\tilde W(v)-\tilde{q}.$
\end{proof}

\begin{remark}
If we consider models involving the restriction $v_3>-m_3$, we need
to consider a variational problem with this constraint, which will
make the admissible set not weakly closed.
\end{remark}

In view of the physical properties of the elastic potential function
$W$ (in particular, the polynomial structure of the
phase-transitional potential function), we can apply the direct method
of the calculus of variation to show existence. In order to deal with
the integral constant $q$, we may restrict the admissible sets (or
choose proper function space) on which we consider the functional or
use Lagrange multiplier to recover it by adding restriction
functional on the original space on which the
functional is defined. \\

In the following, we give some propositions of general nonlinear
functionals. These propositions and further materials can be found
in \cite{De, Ni, ZFC} and the references therein.

\begin{proposition} Let $\mathcal{X}$ be a Banach space, $I$ a real functional defined on $\mathcal{X}$ and $U$ be a
sequentially weakly compact set in $\mathcal{X}$. If $I$ is weakly
lower semi-continuous, then $I$ attains its minimum on $U$, i.e.
there is $x_0\in U$, such that $I(x_0)=\inf_{x\in U}I(x)$.
\end{proposition}

\begin{proof}
Let $c:=\inf_{x\in U}I(x)$. By definition of $\inf$,
there exists $\{x_n\}\subset U$ such that $I(x_n)\rightarrow c$. In
view that $U$ is sequentially weakly compact, $\{x_n\}$ admits a
weakly convergent subsequence, still denoted by $\{x_n\}$. Denote
$x_0\in \mathcal{X}$ the corresponding weak limit. Since $U$ is
weakly closed, we know $x_0\in U$. Noticing that weakly lower
semi-continuity of $I$, we have $c=\lim_{n}I(x_n)\geq I(x_0)$. By
the definition of $c$, we in turn know $I(x_0)=c>-\infty$, which
completes the
proof.
\end{proof}

It is well-known that a bounded weakly closed set in a reflexive
Banach space is weakly compact. In particular, a bounded closed convex
set in reflexive Banach space is weakly compact since weakly close
and close in norm are equivalent for convex sets. Hence we have the
following corollaries:

\begin{corollary} Let $U$ be a bounded weakly closed set in a reflexive Banach space
$\mathcal{X}$ and $I$ be a weakly lower semi-continuous real
functional on $\mathcal{X}$. Then there exists $x_0\in U$ such that
$I(x_0)=\inf_{x\in U}\mathcal{I}$.
\end{corollary}
\begin{definition} A real functional $I$ on a Banach space $\mathcal{X}$ is
said to be \textit{coercive} if
$$\lim_{|x|_{\mathcal{X}}\rightarrow+\infty}I(x)=+\infty.$$
\end{definition}

\begin{corollary} Any coercive weakly lower semi-continuous real functional $I$
defined on a reflexive Banach space $\mathcal{X}$ admits a global
minimizer.
\end{corollary}

\subsection{A general existence result}
In this part, we first give a general result for models with the
physical assumption $\tau_3>0$, i.e, $v_3>-m_3$. We will assume the
following conditions on the potential $W$\\

\textbf{(A1)} $W\in C^2$ and $W(\tau)\rightarrow+\infty$ as
$\tau_3\rightarrow 0^+$. For $\tau_3\leq 0$, define
$W(\tau)=+\infty$;

\textbf{(A2)} There exist a positive constant $C$ such that
$W(\tau)\geq\frac{C}{\tau_3^2}$ for $\tau\in R^n$ ($n=1,2,3$);

\textbf{(A3)} There exists a constant vector $m\in
\mathbb{R}^3_+:=\{m\in\mathbb{R}^3; m_3>0\}$ such that
$\sigma\{D^2W(m)\}\cap\mathbb{R}^1_-\neq\emptyset$. Here
$\sigma\{D^2W(m)\}$ is the spectrum set of $D^2W(m)$.\\

Assumption $(A2)$ implies in particular that the potential is
bounded from below. Assumption $(A3)$ amounts to saying that there is a point
where the potential is concave. From the physical point of view,
this is quite reasonable.

\begin{remark}
A simple kind of potential function is that for an isentropic polytropic gas,
for which  $dW(\tau)=c\tau_3^{-\gamma}$, $\gamma>1$.
This yields $W(\tau)=c_2 \tau_3^{1-\gamma}$, with
$0<1-\gamma<1$ for $\gamma$ in the typical range $1<\gamma<2$
suggested by statistical mechanics \cite{B}, hence blowup as $\tau_3\to 0$ at
rate slower than $c/\tau_3^2$.
Indeed, a point charge model with inverse square law yields in the
continuum limit $W(\tau)\sim \tau_3^{-2/3}$ for dimension $3$,
consistent with a monatomic gas law $\gamma=5/3$.
Thus, (A2) requires a near-range repulsion stronger than inverse square.
Alternatively, one may assume not point charges but particles of finite radius,
as is often done in the literature,
in which case $W(\tau)=\infty$ for $\tau_3\le \alpha$, $\alpha>0$, also
satisfying (A2).  However, in this case, a much simpler argument would suffice
to yield $\tau_3\ge \alpha$ a.e.
\end{remark}

\begin{theorem}\label{6.13}
Assume (A1), (A2) and (A3). If $(\frac{2\pi}{T})^2<\lambda(m)$, then we have a physical
nonconstant periodic wave solution for the problem $(3.2)$ for which
the mean over one period of $\tau$ is $m$. Here $-\lambda(m)$ is the
smallest eigenvalue of $D^2W(m)$.
\end{theorem}

In the following Lemmas of this section, we assume assumption (A1), (A2) and (A3) hold. Define two subsets of $H^1_{T,0}$ by
$$\cA_1:=\{v\in H^1_{T,0}; v_3>-m_3\};$$
$$\cA_2:=\{v\in H^1_{T,0}; v_3\geq-m_3\}.$$

\begin{remark} The admissible set $\cA_1$ is not weakly
closed in $H^1_{T,0}$.
\end{remark}

\begin{lemma}
Under assumptions (A1)-(A3), $\cI$ is a coercive functional on
$H^1_{T,0}$.
\end{lemma}

\begin{proof}
By the definition of $\cI$, we just need to consider
the part $\int_0^T F(v)\,dx$. By assumption (A2), we have

\begin{align*}
\int_0^T F(v)\,dx=&\int_0^T W(v+m)-W(m)-DW(m)\cdot v\,dx\\
=&\int_0^T W(v+m)-W(m)\,dx\\
\geq &-W(m)T>-\infty.
\end{align*}
\end{proof}

By the above lemma, we see that for sufficient large $R$ the
minimizers of $\cI$ on $\cA_i$ are restricted to the sets
$\bar\cA_i:=\cA_i\cap B_{H^1_{T,0}}[0,R]$ for $i=1,2$ where
$B_{H^1_{T,0}}[0,R]$ is the closed ball with center $0$ and radius
$R$ in $H^1_{T,0}$. Define $S_i:=\{v\in\cA_i;
\cI(v)=\inf_{\tilde{v}\in \cA_i}\cI(\tilde{v})\}$. Obviously, we
have $S_i:=\{v\in\bar{\cA_i}; \cI(v)=\inf_{\tilde{v}\in
\cA_i}\cI(\tilde{v})\}$.

\begin{lemma}$\bar{\cA_i}$ is a weakly compact set in $H^1_{T,0}$.
\end{lemma}

\begin{proof}
$\mathcal{\bar A}_2$ is  bounded by its definition. Since $H^1_T$ is
reflexive, we know $\mathcal{\bar A}_2$ is weakly sequentially
compact. Also, $\mathcal{\bar A}_2$ is convex. Indeed, we can use
the definition of convexity of a set to check this easily. An appeal
to Sobolev embedding theorem yields that $\mathcal{\bar A}_2$ is
closed in norm topology of $H_1^T$. For a convex set, closeness in
norm topology and weak topology coincides, hence we have that
$\mathcal{\bar A}_2$ is weakly closed. Putting this information
together, we have shown that $\mathcal{\bar A}_2$ is weakly compact.
\end{proof}

\begin{lemma}$\cI$ is a weakly lower semi-continuous functional on
$H^1_{T,0}$.
\end{lemma}

\begin{proof}
Let $v^{n}\rightarrow v$ weakly in $H^1_{T,0}$. By
Sobolev imbedding, we have $v^n\rightarrow v$ uniformly in $[0,T]$.
Hence we have $\int_0^T F(v^n)\,dx\rightarrow \int_0^T F(v)\,dx$.
Because of the mean zero property, $\int_0^T |v'|^2\,dx$ is of norm
form, hence it is a weakly lower semi-continuous functional.
\end{proof}

\begin{lemma}$S_2\neq\emptyset$ and $v_3\geq-m_3+\epsilon$ for $v\in
S_2$ under the assumption of Theorem 6.12. Here $\epsilon$ is a
positive constant.
\end{lemma}

\begin{proof}
By Proposition 6.8, $S_2\neq\emptyset$. Note that $0\in \cA_2$,
$\cI(0)=0$ and hence $\cI(v)\leq 0$. Hence we will have
$v_3\geq-m_3+\epsilon$. Indeed, suppose there were $x_0\in [0,T]$
such that $v_3 (x_0)=-m_3$. Then by Sobolev embedding there would be
a positive constant $K$ such that
$|v_3(x)+m_3|=|(v_3(x)+m_3)-(v_3(x_0)+m_3)|\leq K|x-x_0|^{1/2}$ for
$x\in [0,T]$. By assumption $(A2)$, we would have
$I(v)=\int_0^T(1/2)|v'|^2\,dx+\int_0^T W(v+m)-W(m)\,dx\geq\int_0^T
CK|x-x_0|^{-1}dx-\int_0^T W(m)dx=+\infty$, a contradiction.
\end{proof}

\begin{lemma}$0\not\in S_1=S_2$ under the assumption of Theorem
6.12.
\end{lemma}

\begin{proof}
Consider the second variation. An easy computation
shows that for $v,\phi$ in $H^1_{T,0}$
$$\cI''(v):(\phi\otimes\phi)=\int_0^T|\phi'|^2\,dx+\int_0^T D^2W(v+m):(\phi\otimes\phi)\,dx.$$
To show $0\not\in S_2$, consider
$$\cI''(0):(\phi\otimes\phi)=\int_0^T|\phi'|^2\,dx+\int_0^T D^2W(m):(\phi\otimes\phi)\,dx.$$
Let $\tilde{\phi}(x)=\eta\sin(\frac{2\pi x}{T})$ for $0<\eta<m_3$
and $v_0\in\mathbb{R}^3$ be a unit eigenvector corresponding to
$-\lambda(m)$. We see that $\phi(x):=\tilde{\phi}(x)v_0\in \cA_2$.
Since $0$ is a critical point of $\cI$ on $H^1_{T,0}$ and
\begin{align*}\cI''(0):(\phi v_0\otimes\phi v_0)=&\int_0^T\eta^2(\frac{2\pi}{T})^2(\cos(\frac{2\pi
x}{T}))^2\,dx-\lambda(m)\int_0^T\eta^2(\sin(\frac{2\pi
x}{T}))^2\,dx\\
=&\frac{\eta^2 T}{2}\{(\frac{2\pi}{T})^2-\lambda(m)\}<0.
\end{align*}
Hence we see that $0\not\in S_2$ and $S_1=S_2$ is obvious.
\end{proof}

\begin{proof}[Proof of Theorem 5.11] Combining Lemma 6.14-Lemma 6.18, we
finish the proof of Therem 6.12. \end{proof}

\begin{remark} \label{conrmk} The condition $(\frac{2\pi}{T})^2<\lambda(m)$
in Theorem 6.12 on the period $T$, is readily seen by Fourier
analysis to be the sharp criterion for stability of the constant
solution $\tau\equiv m$, $u\equiv 0$. Equivalently, it is the Hopf
bifurcation condition as period is increased, marking the minimum
period of bifurcating periodic waves. Thus, it is natural, and no
real restriction.
On the other hand, there may well exist minimizers at whose
mean $m$ $W$ is convex; this condition is sufficient but certainly
not necessary.
Likewise, there exist saddle-point solutions not detected by the
direct approach.
\end{remark}

\subsection{Relation to standard results, and directions for further study}

In the scalar case $\tau\in \RM^1$, the condition that $D^2W(m)$
have a negative eigenvalue is equivalent to convexity of the
Hamiltonian $H$ at the equilibrium $(m,0)$, under which assumption
there are many results on existence of periodic solutions of all
amplitudes; see, for example, \cite{Ra} and later elaborations.
Likewise in the vectorial case $\tau\in \RM^d$, $d>1$, if $D^2
W(m)<0$, then we may appeal to standard theory to obtain existence
of periodic solutions by a variety of means; indeed, the convexity
condition may be substantially relaxed for solutions in the large,
as described in \cite{Ra}, and replaced by global conditions
ensuring, roughly, star-shaped level sets of the Hamiltonian.
On the other hand, review of the potentials considered here reveals that,
typically, it is a single eigenvalue of $D^2W$
that becomes negative and not all eigenvalues,
and so these methods cannot be directly applied.

It is an interesting question to what extent such standard methods could
be adapted to the situation of a Hamiltonian potential (in our
case $-W$) with a single convex mode.
Existence of small amplitude periodic waves at least is treatable by
Hamiltonian Hopf bifurcation analysis.
The question is to what extent if any one can make global conclusions
beyond what we have done here, in particular, to relax
for large solutions the nonconvexity condition on $W$ at $m$.
Finally, it would be interesting to find natural and readily verifiable
conditions for existence of saddle-point solutions in this context.

\section{Existence of periodic solutions for specific models}

In this section, we focus on the existence of periodic waves for the
incompressible models (i.e. models with $\tau_3\equiv1$), namely the
$1D$ models I, II and the 2D incompressible model. First note that
these models do not involve the $\tau_3$ direction. Hence we have no
condition corresponding to $(A1)$. However, the specific
phase-transitional elastic potential energy function $W$ has good
growth rate when $|\tau|\rightarrow+\infty$ for $\tau=\tau_1,\tau_2$
or $(\tau_1,\tau_2)$. This will make our functionals coercive. Hence
we have the following:

\begin{theorem} For the incompressible models, there exist
non-constant periodic standing waves respectively if the mean $m$
(either vector or scalar) satisfies $(A3)$ and
$(\frac{2\pi}{T})^2<\lambda(m)$. When $m$ is scalar, assumption
$(A3)$ means $m$ lies in the elliptic region of the viscoelasticity
system $(2.1)$.
\end{theorem}

\begin{proof}
It is easy to see the corresponding functionals are
coercive, weakly lower semi-continuous functionals on the reflexive
Banach spaces $H^1_{T,0}$. Hence Corollary 6.11 applies. The
verification that the global minimizers respectively are not zero is
entirely the same as in Lemma 6.18 by considering the second
variation.
\end{proof}

Next, we specify the corresponding conditions in Theorem $7.1$ for
these incompressible models.
\subsection{1D Shear Model I}
The condition is
\begin{equation}\label{7.1}
(\frac{2\pi}{T})^2<-8(3m^{2}+1-\varepsilon^{2})\{2(m^{2}+\varepsilon^2)+(m^{2}-\varepsilon^{2})^{2}\}-32m^{2}(m^2+1-\varepsilon^2)^2.
\end{equation}
In particular, if $m=0$, the condition reads
\begin{equation}\label{7.2}
(\frac{2\pi}{T})^2<-8(1-\varepsilon^{2})(2\varepsilon^2+\varepsilon^{4}).
\end{equation}
Condition $(7.2)$ illustrate that our assumption is not a void
assumption. Also, in the mean zero case, $(7.2)$ holds only if
$\varepsilon>1$. Comparing this with the existence result by
phase-plane analysis (see in particular section 8.1),
we see that these results match very well.
\subsection{1D Shear Model II}
The condition is
\begin{equation}\label{7.3}
(\frac{2\pi}{T})^2<-8(3m^{2}+1-\varepsilon^{2})\{2(m^{2}+\varepsilon^2)+(m^{2}-\varepsilon^{2})^{2}\}-32\{m^{2}(m^2+1-\varepsilon^2)^2-\varepsilon^2\}.
\end{equation}
In particular, if $m=0$, the condition reads
\begin{equation}\label{7.4}
(\frac{2\pi}{T})^2<-8(1-\varepsilon^{2})(2\varepsilon^2+\varepsilon^4)+32\varepsilon^2=8(\varepsilon^6+\varepsilon^4+2\varepsilon^2).
\end{equation}
Condition $(7.4)$ implies in particular that for any
$\varepsilon>0$, we have long-periodic oscillatory waves. Similarly,
for any given $T>0$, we have oscillatory waves as long as
$\varepsilon>0$ large enough. Comparing with the phase-plane
analysis (see section 8.2), we see that the related obtained wave phenomena
match very well.

\subsection{2D Incompressible Shear Model}

In this case $D^2W(m)$ is given by its components
\begin{equation*}
w_{11}:=D_{\tau_1\tau_1}W(m)=8(|m|^{2}+1-\varepsilon^{2}+2m_1^2)\{2(|m|^{2}+\varepsilon^2)+(|m|^{2}-\varepsilon^{2})^{2}\}+32m_{1}^{2}(|m|^2+1-\varepsilon^2)^2;
\end{equation*}
\begin{equation*}
w_{12}=w_{21}:=D_{\tau_1\tau_2}W(m)=16m_{1}m_{2}\{2(|m|^{2}+\varepsilon^2)+(|m|^{2}-\varepsilon^{2})^{2}\}
+32m_{1}m_{2}(|m|^{2}+1-\varepsilon^2)^2;
\end{equation*}
\begin{equation*}
w_{22}:=D_{\tau_2\tau_2}W(m)=8(|m|^{2}+1-\varepsilon^2+2m_{2}^2)[2(|m|^{2}+\varepsilon^2)+(|m|^{2}-\varepsilon^2)^2]
+32[m_{2}^{2}(|m|^{2}+1-\varepsilon^2)^2-\varepsilon^2];
\end{equation*}

The corresponding condition is $(\frac{2\pi}{T})^2<\lambda(m)$. This
is obviously a rather mild condition. To see this, we can consider in
particular the mean $m=(m_1, 0)^T$ or $(0,m_2)^T$. Then the results
on the two 1D incompressible models readily give the conclusion
because
we have diagonal matrices.\\

Based on the analysis of these conditions, we have in particular
($m=0$ case):

\begin{theorem} For the 2D shear model, 1D shear model I and II, we have the
following existence result of periodic viscous traveling/standing waves:\\
(1) Given any $\varepsilon>\varepsilon_0\geq0$, for any $T$
satisfying $T> T(\varepsilon)>0$, system
$(5.6)$ hence $(2.1)$ has a nonconstant periodic solution with
some appropriate integral constant $q$; For the 1D model II we have $\varepsilon_0=0$.\\
(2) Given any $T>0$, for any $\varepsilon$ satisfying
$\varepsilon>\varepsilon(T)>0$, system $(5.6)$ hence $(2.1)$ has
a nonconstant periodic solution with some appropriate integral constant $q$.
\end{theorem}

\begin{remark}
From the above theorem, we see in particular indeed for the $2D$
shear model, we have infinitely many nontrivial periodic viscous
traveling waves with appropriate corresponding $q$ values. In particular,
we have a sequence of waves with minimum positive period
$T\rightarrow+\infty$.
\end{remark}

\br\label{realrmk3}
Remarks \ref{realrmk} and \ref{realrmk2} show that
solutions of the specific $3$D model of Section \ref{s:ham}
with $m_1$, $m_2$, and $m_3$ nonzero are genuinely $3$-dimensional
in the sense that they are not confined to a plane in $\tau$-space,
and that solutions of the various $2$D models of Section \ref{s:ham}
are genuinely two-dimensional if the means of both components are nonzero.
That is, we have constructed by the variational approach solutions that
are not obtainable by the planar phase-portrait analysis of the
$1$D case (see
just below).
On the other hand, a dimensional count reveals that, generically,
the periodic solutions nearby a $1$D solution are all $1$D, and
likewise
the periodic solutions nearby a $2$D solution are all $2$D.
\er

\section{1D existence by phase-plane analysis.}\label{s:phase}
In this section, we discuss how to generate periodic waves for
1D models. In $(5.6)$, the integral constant
$q=D_{\tau}W(\tau_-)-\tau_-''$.  Here $(\tau_-, \tau''_-)$ is the
vector evaluated at some specific space value $x_0$. If there indeed
exist periodic-T waves, $q=\frac{1}{T}\int_0^T DW(\tau(x))\,dx$.

By the variational formulation and the usual bootstrap argument, we
conclude that the periodic waves are classical solutions of the
system $(5.6)$. For the possible homoclinic/heteroclinic solutions
of $(5.6)$, standard elliptic regularity theory also guarantees
that
these waves are classical solutions. There must be points (say $x_0,
x_1$) in a period [0,T] such that $\tau'(x_0)=0$ and $\tau''(x_1)=0$
, etc., if such periodic solution did exist for $\tau$ scalar. The
reason is that $\tau (x)$ cannot be always monotone and convex in
view of periodicity (this applies to all derivatives). Hence we can make
the integration constant have the form $q=D_{\tau}W(\tau(x_1))$ for
convenience a priori. Then we can show existence, which in turn guarantees
the a priori assumption. Hence, we could assume
$q=D_{\tau}W(\tau_-)$ to show existence. We adopt this convention in
the following analysis.

The guiding principle is that the ODE systems are planar Hamiltonian
systems. To get complete and clear pictures of the phase-portraits,
we just need to specify the ``potential energy" term
$V(\tau,\tau_-)$ in the Hamiltonian $H(\tau,\tau')$.

\subsection{1D Shear Model I: $\tau_2\equiv0; \tau_3\equiv1$.}

In this section, we denote $\tau=\tau_1$. We use similar
notation in other sections. Recall the elastic potential
$W(\tau)= \Big(2\tau_1^2 + 2\eps^2 +(\tau_1^2-\eps^2)^2 \Big)^2$ and
its first and second order derivatives
$$D_{\tau_1}W(\tau)=8\tau_{1}(\tau_1^2+1-\varepsilon^2)\{2(\tau_1^{2}+\varepsilon^2)+(\tau_1^{2}-\varepsilon^{2})^{2}\};$$
$$D_{\tau_1\tau_1}W(\tau)=8(3\tau_1^{2}+1-\varepsilon^{2})\{2(\tau_1^{2}+\varepsilon^2)+(\tau_1^{2}-\varepsilon^{2})^{2}\}+32\tau_{1}^{2}(\tau_1^2+1-\varepsilon^2)^2.$$

The traveling wave ODE and corresponding Hamiltonian system are

\begin{equation}\label{8.1}
\tau''=W'(\tau)-W'(\tau_-).
\end{equation}
\begin{equation}\label{8.2}
\begin{cases}
\tau'=\tau';\\
\tau''=W'(\tau)-W'(\tau_-).
\end{cases}
\end{equation}
We write the Hamitonian system as follows:
\begin{equation}\label{8.3}
\frac{|\tau'|^2}{2}=H(\tau,\tau')-V(\tau;\tau_-)\equiv
E-V(\tau,\tau_-).
\end{equation}
Here $E$ are constants corresponding to energy level curves of
$H(\tau,\tau')$ and $V(\tau,\tau_-):=q\tau-W(\tau)$.

First, we determine the number of equilibria of the Hamiltonian
system, hence focus on the solution of $W'(\tau)=q$.

Note that $W^{\prime}(\tau)$ is an odd function on the real line,
hence we just need to study its graph on the interval $(0,\infty)$.
In view of the expression of $W^{'}(\tau)$, we need consider the
cases: (1) $0<\varepsilon\leq 1$; (2)  $\varepsilon> 1$.

For the case $0<\varepsilon< 1$, we have $W''(\tau)>0$ for $\tau$
real, hence $W'(\tau)$ is strictly monotone increasing and
$$W'(0)=0;\quad W'(\tau)>0,\quad \mbox{for}\, \tau>0;\quad W'(\tau)<0,\quad \mbox{for}\, \tau<0;$$

Hence in this case, for any given $\tau_-$, the solution of
$W^{\prime}(\tau)=W^{\prime}(\tau_-)$ is $\tau_-$ and unique.

Similar analysis holds true for $\varepsilon=1$. Considering the
definition of $V(\tau,\tau_-)$, we have\\

\begin{proposition} When $0<\varepsilon\leq 1$, $V(\tau;\tau_-)$ has
exactly one critical point $\tau_-$, which must be a global
maximum.\end{proposition}

\begin{remark} In this case, our Hamiltonian system admits no periodic
orbit for any $\tau_-$ (or equivalently, for any $q$).\end{remark}

Next, consider the case $\varepsilon>1$. In this case, we can see
from the expression of $W'(\tau)$ that $W'(\tau)$ has three distinct
zeros: $-\sqrt{\varepsilon^2-1}$, 0, $\sqrt{\varepsilon^2-1}$. A qualitative graph of $W'(\tau_1)$ is as follows:\\
\includegraphics[scale=0.60]{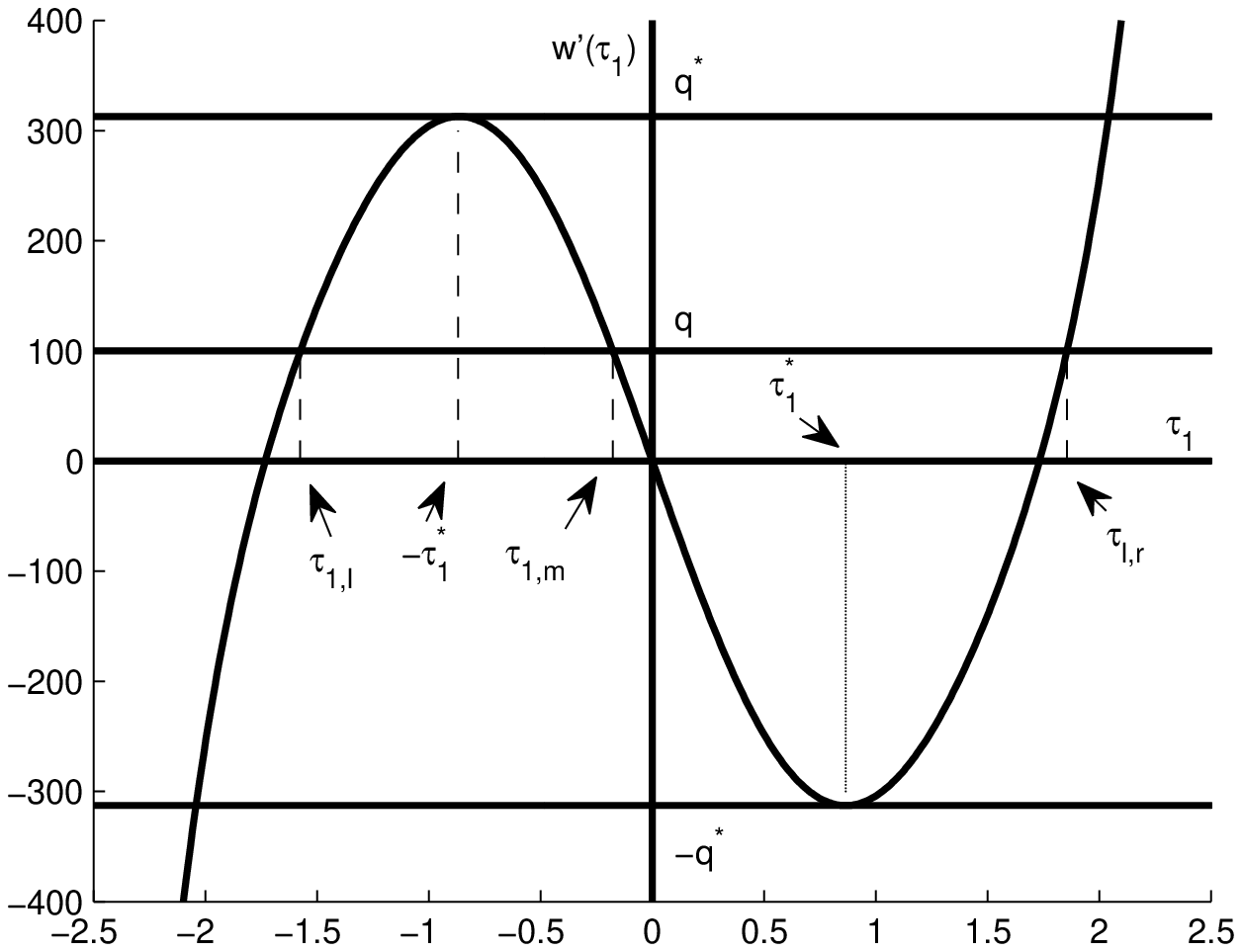}

The corresponding graph for the potential $V(\tau_1,\tau_{1,-})$ is
as follows\\
\includegraphics[scale=0.60]{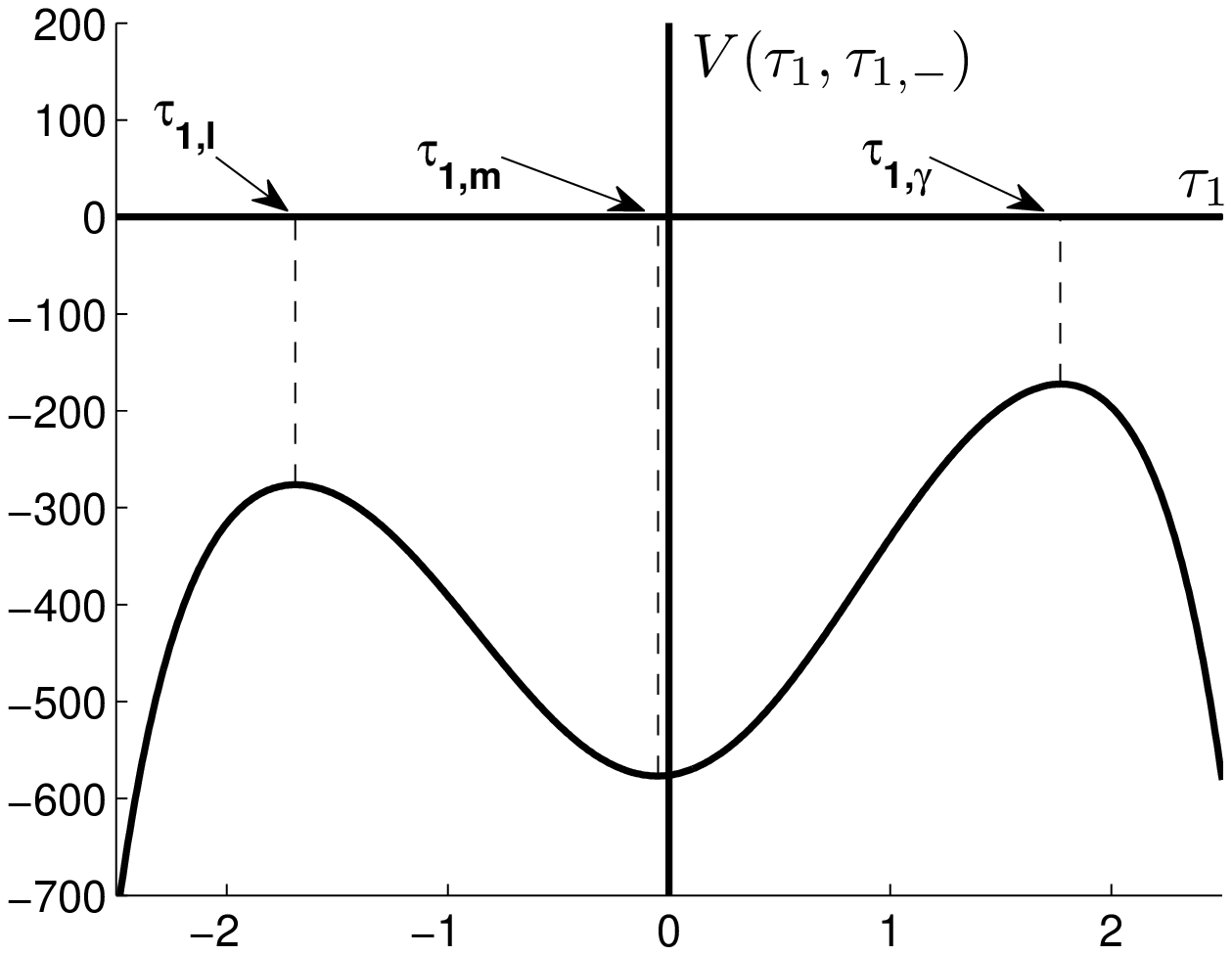}

\begin{proposition} The function $W'(\tau)$ has exactly two critical
points.\end{proposition}

\begin{proof}
By symmetry, we do the following computations: Denote
$\tau_1^2:=X$ and $\varepsilon^2:=a>1$. We want to show that the function

$$f(X):=(3X+1-a)[2(X+a)+(X-a)^2]+4X(X+1-a)^2$$
has exactly one zero when $X>0$.

First, noting that $f(0)<0$ and $f(\frac{a-1}{3})>0$, we know
that $f(X)$ has a root on $(0,\frac{a-1}{3})$. Also note that $f(X)>0$ on
$[\frac{a-1}{3},\infty)$, hence we just need to show that $f(X)$ admits
a unique zero on $(0,\frac{a-1}{3})$. Computing the derivative, we
have

$$f'(X)=3[7X^2+10(1-a)X+a^2+2a+2(1-a)^2].$$

Denote
$$\Delta=100(1-a)^2-28[a^2+2a+2(1-a)^2].$$

If $\Delta\leq 0$, we know that $f'(X)\geq 0$, hence $f(X)$ is monotone
increasing, which implies that $f(X)$ admits a unique zero;

If $\Delta>0$, we will have two positive roots for $f'(X)=0$ and the
smaller one is $\frac{10(a-1)-\sqrt{\Delta}}{14}$. However, we can
show that $\frac{10(a-1)-\sqrt{\Delta}}{14}\geq\frac{a-1}{3}$, hence the
function $f(X)$ is monotone increasing on the interval
$(0,\frac{a-1}{3})$, which also implies the uniqueness of the zero.
\end{proof}

Now we have a clear picture on the potential $W'(\tau)$ (see the graph for $W'(\tau_1)$).\\

\begin{proposition} When $\varepsilon>1$, the function $W'(\tau)$ is
an odd function with 3 zeros and 2 critical points and goes to
infinity when $\tau\rightarrow+\infty$.\end{proposition}

Denote the two critical values
of $W'(\tau)$ as $q^*>0$ and $-q^*$, for convenience denoting $Q=q^*$.
Then we have the following property:\\

\begin{proposition} Assume $\varepsilon>1$. When $|q|> Q$, the
equation $W'(\tau)=q$ has exactly one solution; When $|q|= Q$, the
equation $W'(\tau)=q$ has exactly 2 solutions; When $|q|< Q$, the
equation $W'(\tau)=q$ has exactly 3 solutions.\end{proposition}

As the solutions of $W'(\tau)=q$ correspond to the critical points
of $V(\tau;\tau_-)$, we have:\\

\begin{theorem} For $|q|\geq Q$, the Hamiltonian system admits
no periodic orbit; For $|q|<Q$, the Hamiltonian system admits a
family of nontrivial periodic orbits. Further if $q=0$, the
Hamiltonian system also admits a heteroclinic orbit.\end{theorem}

\begin{proof}
For $|q|\geq Q$, we know that $V(\tau;\tau_-)$ has a global maximum
and hence the Hamiltonian admits no periodic orbit. For the case
$|q|<Q$, we know that the potential $V(\tau;\tau_-)$ must have
exactly 3 critical points with 2 local maxima and 1 local minimum.
Also, we know that $V(\tau;\tau_-)$ has strictly lower energy at the
local minimum than at the two local maxima. Hence the existence of a
family of periodic orbits follows. Further, when $q=0$, the energies
at the two local maxima of $V(\tau;\tau_-)$ are the same, hence we
get an heteroclinic orbit.
\end{proof}

\begin{remark}
We may compare the two energy values of
$V(\tau;\tau_-)$ at the two local maxima. If they are equal (when
$q=0$ in particular), we have a heteroclinic orbit. In general, they
are not equal to each other, which yields a homoclinic
orbit.
\end{remark}

\subsection{1D Shear Model II: $\tau_1\equiv0; \tau_3\equiv1$.}

Recall the elastic potential
$$W(\tau)= \Big(2(\tau_2-\varepsilon)^2+(\tau_2^2-\eps^2)^2
\Big)\times \Big(2(\tau_2+\varepsilon)^2+(\tau_2^2-\eps^2)^2\Big)$$
and the relevant derivatives
\begin{equation*}
D_{\tau_2}W(\tau)=8\tau_{2}(\tau_2^2+1-\varepsilon^2)\{2(\tau_2^{2}+\varepsilon^2)+(\tau_2^{2}-\varepsilon^{2})^{2}\}-32\tau_2\varepsilon^2;
\end{equation*}
\begin{equation*}
D_{\tau_2\tau_2}W(\tau)=8(3\tau_2^{2}+1-\varepsilon^{2})\{2(\tau_2^{2}+\varepsilon^2)+(\tau_2^{2}-\varepsilon^{2})^{2}\}+32\{\tau_{2}^{2}(\tau_2^2+1-\varepsilon^2)^2-\varepsilon^2\}.
\end{equation*}

As above, we list the Hamiltonian system and the potential
$V(\tau;\tau_-)$.  The system is:

\begin{equation}\label{8.4}
\begin{cases}\tau'=\tau';\\
\tau"=W'(\tau)-W'(\tau_-).
\end{cases}
\end{equation}

The potential is $V(\tau;\tau_-)=q\tau -W(\tau)$.

\begin{remark} There is a slight difference with the $1D$ shear model I in
the function $W'(\tau)$. Because of this difference, we do not need
to restrict the positive number $\varepsilon$ to get periodic orbits
for the parameter $q$ in proper range.  The conclusions are
completely the same when $1\geq\varepsilon>0$ as in $1D$ shear model
I when $\varepsilon>1$.\end{remark} We have the following:

\begin{proposition} When $1\geq\varepsilon>0$, the behavior of the function
$W'(\tau)$ is the same as that of the function $W'(\tau)$ in the $1D$ shear
model I when $\varepsilon>1$. In fact, $W''(\tau)$ is monotone
increasing in this case for $\tau>0$.\end{proposition}

\begin{remark} For the range $\varepsilon>1$, numerics suggest
that the behaviors are also the same as we may show that the
function $W'(\tau)$ has exactly three solutions and two critical
points. We have a small problem to verify this by direct computation
though we just need to show that $f(X)>0$ evaluated at the larger
root of $f'(X)$ ($f(X)$ is defined similar as in $1D$ shear model I
as the second derivatives of the two potentials differ with a
constant $32\varepsilon^2$. Even without this, we still can conclude
the existence of periodic orbits since the potential
$V(\tau;\tau_-)$ admits a minimum. Together with the existence
obtained by variational argument, we know that there are still
infinitely many nontrivial periodic waves for any
$\varepsilon>0.$\end{remark}

\subsection{1D Compressible Model III} For this model, we need to pay
special attention to the physical restriction $\tau_3>0$
when we do the phase-plane analysis. To find physical waves, we use
a continuity argument and a simple comparison criterion.

In this case $\tau_1=\tau_2\equiv 0$; let $\tau=\tau_3$. The
potential and its derivatives are given below. The elastic potential
becomes $W(\tau)= \Big(2\eps^2 +(\tau_3^2-1-\eps^2)^2 \Big)^2$. Its
first and second order derivative are
$D_{\tau_3}W(\tau)=8\tau_{3}(\tau_3^{2}-1-\varepsilon^{2})\{2\varepsilon^2+(\tau_3^{2}-1-\varepsilon^{2})^{2}\}$
and
\begin{equation*}
w_{33}:=D_{\tau_3\tau_3}W(\tau)=8(3\tau_3^2-1-\varepsilon^2)\{2\varepsilon^2+(\tau_3^{2}-1-\varepsilon^{2})^{2}\}
+32\tau_{3}^2(\tau_3^{2}-1-\varepsilon^2)^2.
\end{equation*}

We write $V=V(\tau,q,\varepsilon):=q\tau-W(\tau)$ in this section to
emphasize the analytical dependence of $V$ on the parameters $q$ and
$\varepsilon$ (because $V$ is a polynomial). As in previous sections,
we see:
\begin{proposition}
(1)$ W'(\tau)= q$  always has one, two or three roots when $|q|>Q$,
$|q|=Q$ or $|q|<Q$ for some positive $Q$. In the case that $ W'(\tau)= q$
has 3 distinct roots, we denote them from small to large by $\tau_l,
\tau_m$ and $\tau_r$.

(2) $W'(\tau)$ has exactly two critical points.
\end{proposition}

As before, in order to analyze the existence of periodic or
homoclinic/heteroclinic waves, we just need to consider the
potential energy $V(\tau, q,\varepsilon)$. Further, in order to have
physical waves, we need necessarily that $ - Q < q <
 0 $. In this situation, the two roots $\tau_m$ and $\tau_r$ of $W' (\tau)=q$ are positive.
Noticing that $\tau=\tau_m$ is a local minimizer of
$V(\tau,q,\varepsilon)$, there is a periodic annulus around
$\tau_m$. Hence we have the following proposition:

\begin{proposition}
When $ - Q< q < 0 $, there always exists a periodic annulus.
\end{proposition}

To show existence of physical homoclinic orbit, we just need to
compare the values of $V(0, q,\varepsilon)=:V(0)$ and $V (\tau_r,
q,\varepsilon)=:V(r)$. We have

\begin{proposition}
Let $-Q<q<0$. If $V (0) > V (r)$, there is a physical homoclinic
orbit; If $V (0)\leq V (r)$, there is no physical homoclinic orbit.
\end{proposition}

In particular, for the case $q =0$, the 3 distinct roots of
$W'(\tau)=q$ are easily seen to be $\tau_l=-\sqrt{1+\varepsilon^2},
\tau_m=0$ and $\tau_r=-\tau_l$. So $V (0,
0,\varepsilon)=-(2\varepsilon^2 + (1 + \varepsilon^2)^2)^2 < V
(\tau_r, 0, \varepsilon)=- (2\varepsilon^2)^2$. By continuity, we
have the following conclusion:
\begin{proposition}
There exists a constant $\eta > 0$ such that if $ - \eta < q\leq
0 $, then there exists no physical homoclinic orbit.
\end{proposition}

\begin{proof}
When $q = 0 $, $V (0, 0,\varepsilon)<V (\tau_r,
0, \varepsilon)$. Thus, by continuous dependence and Propositions 8.11 and
8.13, we have the relation $V(0)<V(m)$ holds when $q < 0 $ is small and the conclusion holds.
\end{proof}

Next, we study the existence of physical homoclinic waves when
$-Q<q<0$ is large. For this purpose, we first set $\varepsilon=0$
then proceed by a perturbation argument. When $\varepsilon=0$, the
corresponding elastic energy function and its derivatives are:
$$\tilde{W}(\tau)=(\tau^2-1)^4;$$
$$\tilde{W}'(\tau)=8\tau(\tau^2-1)^3;$$
$$\tilde{W}''(\tau)=8(\tau^2-1)^2(7\tau^2-1).$$
Note that $\tilde{W}''(\tau)=0$ has roots $\tau=\pm1,\pm\sqrt{1/7}$
(this can be easily computed).

For this potential
$\tilde{V}(\tau,q,\varepsilon=0):=q\tau-\tilde{W}(\tau)$, we need
$|q|<\tilde{W}'(-\sqrt{1/7})=8(\frac{6}{7})^3\sqrt{\frac{1}{7}}$ to
have a homoclinic wave. For physical ones, we need
$-8(\frac{6}{7})^3\sqrt{\frac{1}{7}}<q<0$. Consider the case
$q\rightarrow-8(\frac{6}{7})^3\sqrt{\frac{1}{7}}$ from the right, we
see the largest root $\tau_m$ of $\tilde{W}'(\tau)=q$ tends to
$\sqrt{1/7}$. Consequently,
$\tilde{V}(\tau=0,q,\varepsilon=0)\rightarrow-\tilde{W}(0)=-1$ and
the right local maximum value of
$\tilde{V}(\tau_m,q,\varepsilon=0)\rightarrow-8(\frac{6}{7})^3\sqrt{\frac{1}{7}}-\tilde{W}(\sqrt{1/7})<-1$.
By Proposition 8.11 and Proposition 8.13, we have:

\begin{proposition} For the compressible 1D model, assume $-Q<q<0$ ($Q$ as
in Proposition 8.11).
Then, when
$\varepsilon>0$ and $q+Q$ are small, we have a physical homoclinic
orbit.
\end{proposition}

\section{Time-evolutionary stability}\label{s:stability}
We conclude by discussing briefly the question of time-evolutionary
stability of elastic traveling waves with strain-gradient effects.

\subsection{Spectral vs. nonlinear stability}

A very useful observation regarding the earlier phenomenological
models
$ \tau_t-u_x=0$, $u_t+ dW(\tau)_x= b u_{xx}-d\tau_{xxx}$,
$b,d>0$ constant, for $1$-D shear flow, made by Schecter and Shearer in
\cite{SS}, was that for a wide range of $b$, $d$, specifically,
$d<b^2/4$, the system
can be transformed by the change of independent variable
$u\to \tilde u:= u-c\tau_x$, $c(b-c)=d$ to the fully parabolic system
$$
\begin{aligned}
 \tau_t-\tilde u_x&=c \tau_{xx},\\
 \tilde u_t+ dW(\tau)_x&= (b-c) \tilde u_{xx},
\end{aligned}
$$
thus allowing the treatment of nonlinear stability by standard parabolic
techniques, taking into account, for example, sectorial structure, parabolic
smoothing, etc.

Quite recently, this observation has been profoundly generalized by
M. Kotschote \cite{K}, who showed that a somewhat different
transformation in similar spirit may be used to convert elasticity
or fluid-dynamical equations with strain-gradient (resp.
capillarity) effects to quasilinear fully parabolic form, {\it in
complete generality}, not only to the cases $4d>b^2$ previously
uncovered for the phenomenological model, but to the entire class of
physical models considered here. For further discussion/description
of this transformation, see Appendix \ref{s:structure}.

This reduces the question of nonlinear stability to a standard format
already well studied.
In particular, it follows that (except possibly in nongeneric boundary cases
of neutrally stable spectrum) {\it nonlinear stability is equivalent to
spectral stability,} appropriately defined.
This follows for heteroclinic and homoclinic waves by the analysis
of \cite{HZ}, and for periodic waves by the analysis of
\cite{JZ}.\footnote{
The analysis of \cite{JZ} concerns modulational stability,
or stability with respect to localized
perturbations on the whole line; co-periodic stability may
be treated by standard semigroup techniques \cite{He}.
Spectral analyses of \cite{OZ,BYZ,PSZ} suggest that modulational
stability occurs rarely if ever for viscoelastic waves.
}
For precise definitions of the notions of spectral stability, we
refer the reader to those references; in the shock wave (heteroclinic
or homoclinic) case, see also the discussion of \cite{BLeZ}.
Spectral stability may be efficiently determined numerically by
Evans function techniques, as in for example \cite{BHRZ,BLeZ,BHZ,BJNRZ1,BJNRZ2}.
We intend to carry out such a numerical study
in a followup work \cite{BYZ}.

\subsubsection{Transformation to strictly parabolic form.}\label{s:structure}

We now show how to apply the approach of Kotschote in our context
and verify that we thereby obtain the structural properties needed
to apply the general theory of \cite{HZ}. So we shall in the
following verify the structural properties of the elastic model with
strain-gradient effect and related modified systems obtained by the
apporach of Kotschote \cite{K}. To be clear, we collect these
related systems. The original system is
\begin{equation}\label{B1}
\begin{cases}\tau_t -u_x=0\\
u_t +\sigma(\tau)_x=(b(\tau)u_x)_x - (d(\tau_x)
\tau_{xx})_x.\end{cases}
\end{equation}

Here $\sigma=-D_{\tau} W(\tau)$, $d(\cdot)=D^2\Psi(\cdot)$,
$d(\cdot)=Id$ and
 \be\label{b3} b(\tau)=\tau_3^{-1}\bp 1&0&0\\0&1&0\\0&0&2\ep\ee

Introducing the phase variable $z:=\tau_x$, we may write $(9.1)$ as
a quasilinear second-order system
\begin{equation}\label{9.2}
\begin{cases}\tau_t +z_x -u_x=\tau_{xx}\\
z_t = u_{xx}\\
u_t +\sigma(\tau)_x=(b(\tau)u_x)_x - (d(z) z_{x})_x.\end{cases}
\end{equation}

\begin{remark} This transformation, introduced in \cite{K}, is similar in
spirit to but more general than the one\footnote{A transformation
$(\tau, u)\to (\tau, u-c \tau_x)$ reducing the model to a parabolic
system of the same size.} introduced by Slemrod \cite{S1,S2,S3} and
used in \cite{OZ} for an artificial viscosity/capillarity model.
\end{remark}

We can slightly modify the above system in the second equation. Then
we have the following system

\begin{equation}\label{B3}
\begin{cases}
\tau_t +z_x -u_x=\tau_{xx}\\
z_t +z_x= u_{xx}+\tau_{xx}\\
u_t +\sigma(\tau)_x=(b(\tau)u_x)_x - (d(z) z_{x})_x.
\end{cases}
\end{equation}
If we write the above system (9.3) and (9.4) in matrix form
$$U_t+f(U)_x=(B(U)U_x)_x$$
using the variable $U:=\begin{pmatrix}\tau \\ z \\ u\end{pmatrix}$,
the corresponding matrix $B$ becomes:
$$\bp I_3&0_3&0_3\\ 0_3&0_3&I_3\\0_3 & -d(z) &b(\tau)\ep;$$
$$\bp I_3&0_3&0_3\\ I_3&0_3&I_3\\0_3 & -d(z) &b(\tau)\ep.$$
Also, the corresponding matrix $Df(U)$ for the two systems are:
$$\begin{pmatrix}
0_3 & I_3 & -I_3\\
0_3 & 0_3 & 0_3\\
D\sigma(\tau) & 0_3 & 0_3
\end{pmatrix};$$

$$\begin{pmatrix}
0_3 & I_3 & -I_3\\
0_3 & I_3 & 0_3\\
D\sigma(\tau) & 0_3 & 0_3
\end{pmatrix}.$$

\begin{proposition} (Strict parabolicity) Systems $(9.3)$ and $(9.4)$ are both strictly parabolic systems in the sense that the spectrum of $B$ have positive real parts.\end{proposition}

\begin{proof}
Comparing the two matrices $B$ above, we know they have the same
spectrum. We prove this proposition for $\tau, z, u\in
\mathbb{R}^3$. The lower dimension cases becomes easier and the
computations are totally the same . Pick one of the $B's$, say
$$\bp I_3&0_3&0_3\\ 0_3&0_3&I_3\\0_3 & -d(z) &b(\tau)\ep.$$ To
compute the spectrum of $B$, consider the characteristic polynomial
$\det(\lambda I_9-B)=0$, which is

$$\det\bp (\lambda-1) I_3&0_3&0_3\\0_3&\lambda I_3&-I_3\\0_3 & +d(z) &\lambda-b(\tau)\ep=0.$$
Doing Laplace expansion and elementary column transformation , we
get
$$\det\{(\lambda-1)I_3\}\det\{(\lambda^{2}I_3-\lambda b(\tau)+d(z))\}=(\lambda-1)^3(\lambda^2-\frac{\lambda}{\tau_3}+1)^2(\lambda^2-\frac{2\lambda}{\tau_3}+1)=0.$$

From the first factor of the above degree 9 polynomial, we get three
equal root 1 which has positive real parts . The other 6 roots also
have positive real parts noticing that $\tau_3>0$.
\end{proof}

\begin{proposition} (Nonzero characteristic speeds) The corresponding first order systems of (C.4) has nonzero characteristic speed at $\tau$ where the matrix $D^2W(\tau)$ are strictly positive definite.\end{proposition}

\begin{proof}
Again, we prove this for $\tau,z,u\in\mathbb{R}^3$. To prove the
corresponding first order system is non-characteristic, we consider
the spectrum of the matrix $\begin{pmatrix}0_3 & I_3 & -I_3\\ 0_3 &
I_3 & 0_3\\D\sigma(\tau) & 0_3 & 0_3\end{pmatrix}$. We get the
following system for the characteristic speed:

$$\det(\lambda I_9-Df(U))=\begin{pmatrix}\lambda I_3 & -I_3 & I_3\\
0_3 &(\lambda-1)I_3 & 0_3\\-D\sigma(\tau) & 0_3 & \lambda
I_3\end{pmatrix}=0.$$ By direct computation, we know:
\begin{equation}\det(\lambda
I_9-Df(U))=\det\Big((\lambda-1)I_3\Big)\det\Big(-D\sigma(\tau)-\lambda^2
I_3\Big)=0.\end{equation} It is easy to see we have three roots 1
which is not 0. The other roots satisfy the algebraic equation:
$\det\Big(-D\sigma(\tau)-\lambda^2
I_3\Big)=\det\Big(D^2W(\tau)-\lambda^2 I_3\Big)=0$. Hence the
proposition follows.
\end{proof}

\begin{proposition} (Same spectrum) For system $(9.4)$ and its first order system, the matrix $Df(U)$ and $B^{-1}Df(U)$ have the same spectrum.\end{proposition}

\begin{proof}
We prove this for the variables $\tau, z, u\in \mathbb{R}^3$. It is
easy to verify that
$$B^{-1}=\begin{pmatrix}
I_3 & 0_3 & 0_3\\
-b(\tau) & b(\tau) & -I_3\\
-I_3 & I_3 & 0_3
\end{pmatrix}.$$
Since
$$Df(U)=\begin{pmatrix}
0_3 & I_3 & -I_3\\
0_3 & I_3 & 0_3\\
D\sigma(\tau) & 0_3 & 0_3
\end{pmatrix},$$
we immediately get
$$B^{-1}Df(U)=\begin{pmatrix}
0_3 & I_3 & -I_3\\
-D\sigma(\tau) & 0_3 & b(\tau)\\
0_3 & 0_3 & I_3
\end{pmatrix}.$$
Considering the corresponding eigenvalue problem, we have:
$$\det\Big(\lambda I_9-B^{-1}Df(U)\Big)=\det\begin{pmatrix}
\lambda I_3 & -I_3 & I_3\\
D\sigma(\tau) & \lambda I_3 & -b(\tau)\\
0_3 & 0_3 & (\lambda-1)I_3
\end{pmatrix}=0.$$

Doing Laplace expansion and performing basic transformation, we get:
$$\det\Big(\lambda I_3-I_3 \Big)\det\Big(\lambda^2 I_3+D\sigma(\tau)
\Big)=\det\Big(\lambda I_3-I_3 \Big)\det\Big(\lambda^2
I_3-D^2W(\tau) \Big)=0,$$ which implies the conclusion by noticing
$(9.5)$.
\end{proof}

\subsection{Variational vs. time-evolutionary stability}

More fundamentally, perhaps, there is a relation between variational
stability of periodic waves and their time-evolutionary stability as
solutions of \eqref{2.1}. In particular, the energy functional that
we minimized in constructing periodic solutions is essentially the
self-same functional that defines the mechanical energy of the
system, a Lyapunov functional that decreases with the flow of
\eqref{2.1}; as we show below in Section 9.2.1. This gives a strong
link between the two notions of stability. Indeed, it can be used to
directly show that the periodic waves constructed as minimizers of
the associated variational problem are time-evolutionarily stable
with respect to co-periodic perturbations (Remark \ref{rmkC2}).

Moreover, as discussed further in \cite{BYZ}, the indirect spectral arguments
of \cite{Z2} on $1$-D heteroclinic and homoclinic waves extend to
to the general-dimensional and or periodic case, yielding the much
stronger result that variational stability in each of these contexts
is {\it necessary and sufficient} for time-evolutionary stability
(co-periodic stability, in the case of periodic waves, and variational
stability constrained by a prescribed mean).
Moreover, these arguments yield at the same time the curious fact that
unstable spectra of the linearized operator about the wave {\it must, if it
exists, be real}.
These properties give additional insight,
and additional avenues by which time-evolutionary stability
may be studied.

{\it In particular, this shows that
the waves we have constructed are the} (co-periodically) {\it stable ones.}
However, these are not necessarily the only stable waves, as we
did not construct all minimizers of the variational problem,
but only those with mean satisfying a nonconvexity condition.
Just recently (in particular, after the completion of the analysis
of this paper),
there has been introduced in \cite{PSZ} a different, more direct
argument showing equivalence of variational and time-evolutionary
stability, which yields at the same time concise conditions for
variational stability.
These yield in particular that the sharp condition for stability
is not the condition of nonconvexity of $W$ at $m$,
defined as the mean over one period of $\tau$, but rather
the ``averaged'' condition of nonconvexity of the
Jacobian with respect to $m$ of the mean over one period of $DW(\tau)$.
See \cite{PSZ} for further details.

\subsubsection{Calculus of Variation and relative mechanical
energy}\label{s:rel}

In this section, we discuss further the relation between variational
stability and time-evolutionary stability with respect to
co-periodic perturbations of periodic waves. It is well-known that
the physical foundation of calculus of variation is the principle of
least action (also known as Hamilton's principle, Maupertuis'
principle, see \cite{Lan, MW, PB} and the references therein). Hence
the functional in calculus of variation represents some sort of
``energy" accordingly.

To be self-contained, we recall the system \eqref{3.2} (which is
\eqref{C1} below now)
\begin{equation}\label{C1}
\begin{cases}\tau_t-u_x=0;\\
u_t+\sigma(\tau)_x=(b(\tau)u_x)_x-(d(\tau_x)\tau_{xx})_x
\end{cases}
\end{equation}

Here the functions $b,d$ are the same as before. From [BLeZ] we know
the following is the associated entropy in the sense of hyperbolic
system of conservation laws:
\begin{equation}\label{C2}
\eta(\tau,u)=\frac{u^2}{2}+W(\tau)+\Psi(\tau_x).
\end{equation}

Consider the following mechanical energy for a given positive period
$T$

\begin{equation}\label{C3}
\mathcal{E}=\int_0^T\eta(\tau,u)dx=\int_0^T\frac{u^2}{2}+W(\tau)+\Psi(\tau_x)\,dx
\end{equation}

In any case, the mechanical energy $\mathcal{E}$ decreases along the
flow for periodic boundary condition, with dissipation as follows
\begin{align*}
\frac{d}{dt}\mathcal{E}(\tau,u)=&\int_0^T
\big(uu_t+DW(\tau)\tau_t+D\Psi(\tau_x)\tau_{xt}\big)dx\\
=&\int_0^T -u_x b(\tau)u_x+u_x d(\tau_x)\tau_{xx}+u_x
\sigma(\tau)+DW(\tau)u_x+D\Psi(\tau_x)u_{xx})dx\\
=&-\int_0^T u_x b(\tau)u_x dx\leq 0.
\end{align*}
Hence we see immediately that periodic traveling wave solutions must
have $u\equiv\mbox{constant vector}$, from which we then find easily
the speed $s=0$ in view of the relation $u_t=-su_x=0$ and
$\tau_t-u_x=0$.\\

Next, we adopt the periodic Sobolev space framework to discuss the
relation between the least action functional $(6.5)$ and the
\textit{relative mechanical energy}
\begin{equation}\label{C4}
\mathcal{E}(\tau,u;<\tau>,
<u>):=\mathcal{E}(\tau,u)-\mathcal{E}(<\tau>,
<u>)-D\mathcal{E}(<\tau>, <u>)\cdot(\tau, u),
\end{equation}
where $<\tau>:=\frac{1}{T}\int_0^T \tau(x)\,dx$ as in physics
literature and similarly for $u$.

After a brief computation, we get that the relative mechanical
energy (``relative entropy" in the sense of hyperbolic systems of
conservation law) is given by
\begin{align*}
\mathcal{E}(\tau,u; <\tau>,<u>)=&\int_0^T
\Psi(\tau_x)+W(\tau)-DW(<\tau>)\cdot\tau-W(<\tau>)\,dx\\+&\int_0^T\frac{1}{2}|u|^2-\frac{1}{2}|<u>|^2-<u>\cdot
u\,dx.
\end{align*}

From the system \eqref{3.2} we know that the structure is preserved
under the transformation $u\rightarrow u+c$ where $c$ is an
arbitrary constant vector. Hence without loss of generality, we can
let $u\equiv 0$. Hence we get the following expression by further
choosing $<u>=m$ as before
\begin{align*}
\mathcal{E}(\tau,0; 0,m)=&\int_0^T
\Psi(\tau_x)+W(\tau)-DW(m)\cdot\tau-W(m)\,dx\\
=&\int_0^T \frac{1}{2}|\tau_x|^2+W(\tau)-DW(m)\cdot\tau-W(m)\,dx
\end{align*}

Defining the translated variable $v(x)=\tau(x)-m$, we get the
relation between the least action functional and the relative
entropy

$$\mathcal{E}(v,0; 0,m)=\mathcal{I}(v)+\mbox{constant},$$
where the constant is given by $T(DW(m)\cdot m)$.

\begin{remark} The relative entropy in the sense of hyperbolic system of conservation law is a rather common construction,
meant to be stationary about the reference configuration (in our
case $<\tau>=m$ and $<u>=0$).
\end{remark}

\begin{remark}\label{rmkC2}
The discussion above also sheds some light on the relation between
time-evolutionary properties and variational structure of the
Hamiltonian structure of our problem (See \cite{Z2} for further
discussions). By the method of \cite{GSSI, GSSII}, we see that the
mean-constraint minimizers we constructed are necessarily
\textit{stable in the time-evolutionary sense} with respect to
co-periodic perturbation if they are stable in the variational
sense.
\end{remark}

\appendix
\section{Appendix: Phase-transitional elasticity}\label{s:phase_app}

In this appendix, we collect some computations for the
phase-transitional elasticity (\cite{BLeZ},\cite{FP} and references
therein).\footnote{For the convenience of the reader, we note that the the
vector $(a_1,a_2,a_3)$ in \cite{BLeZ} corresponds to
$(\tau_1,\tau_2,\tau_3)$ here.}
$$
W(F):= \Big| F^T F -C_-\big|^2 \Big| F^T F -C_+\big|^2,
$$
where
$$
C_\pm= (F^T F)_\pm:=
\begin{pmatrix}
1 & 0 & 0\\
0 & 1 & \pm \eps\\
0 & \pm \eps & 1+\eps^2\\
\end{pmatrix},
\qquad F_\pm=
\begin{pmatrix}
1 & 0 & 0\\
0 & 1 & \pm \eps\\
0 & 0 & 1\\
\end{pmatrix}.
$$
Evidently, $W$ is minimized among planar deformations at the two
values $ A_\pm=
\begin{pmatrix}
0 & 0 & 0\\
0 & 0 & \pm \eps\\
0 & 0 & 0\\
\end{pmatrix}.
$ Indeed, we have then
$$
W(F)= \Big(2\tau_1^2 + 2(\tau_2-\eps)^2 +(|\tau|^2-1-\eps^2)^2
\Big)\Big(2\tau_1^2 + 2(\tau_2+\eps)^2 +(|\tau|^2-1-\eps^2)^2 \Big),
$$
where, as in \cite{BLeZ}, $ A=
\begin{pmatrix}
0 & 0 & \tau_1\\
0 & 0 & \tau_2\\
0 & 0 & \tau_3\\
\end{pmatrix}$
and $F=
\begin{pmatrix}
1 & 0 & \tau_1\\
0 & 1 & \tau_2\\
0 & 0 & \tau_3\\
\end{pmatrix},
$
so that
\be\label{ftf}
F^TF= \begin{pmatrix}
1 & 0 & \tau_1\\
0 & 1 & \tau_2\\
\tau_1 & \tau_2 & |\tau|^2\\
\end{pmatrix}
\ee
and
\be\label{ftfcalc} W(F)= \Big|\begin{pmatrix}
0 & 0 & \tau_1\\
0 & 0 & \tau_2+\eps\\
\tau_1 & \tau_2+\eps & |\tau|^2-1 -\eps^2\\
\end{pmatrix}
\Big|^2 \Big|\begin{pmatrix}
0 & 0 & \tau_1\\
0 & 0 & \tau_2-\eps\\
\tau_1 & \tau_2-\eps & |\tau|^2-1 -\eps^2\\
\end{pmatrix}
\Big|^2
\ee

\be\label{ftf2} (F^TF-C_{\pm})^2= \begin{pmatrix}
\tau_{1}^2 & * & *\\
* & (\tau_2 \mp\varepsilon)^2 & *\\
* & * &
\tau_{1}^{2}+(\tau_{2}\mp\varepsilon)^{2}+(|\tau|^2-1-\varepsilon^2)^2\\
\end{pmatrix}.
\ee

\section*{Acknowledgement.} The author is grateful to his thesis
advisor, Professor Kevin Zumbrun,  for the illuminating discussions,
precious guidance and encouragement during this work. He would like
also to thank Blake Barker for helpful discussions.

\end{document}